\documentclass[leqno]{cmslatex}

\usepackage{latexsym, amsfonts}
\usepackage{amssymb}
\usepackage{amsmath}

\renewcommand{\theequation}{\arabic{section}.\arabic{equation}}

\def\a{a}
\def\b{b}

\def\eps{{\epsilon}}
\def\bfR{\mathbb{R}}
\def\bfN{\mathbb{N}}

\def\mcL{\mathcal{L}}
\def\mcV{\mathcal{V}}
\def\esssup{\text{ess sup}}

\newcommand{\del}{\partial} 
\newcommand{\norm}[1]{\ensuremath{\left\| #1 \right\|}}
\newcommand{\ip}[2]{\ensuremath{\left\langle #1, #2 \right\rangle}}
\newcommand{\abs}[1]{\ensuremath{\left|#1\right|}}
\renewcommand{\grad}{\nabla}

\everymath{\displaystyle}

\begin{document}

\title{Global Existence for the ``One and one-half'' dimensional relativistic Vlasov-Maxwell-Fokker-Planck system \thanks{This
work was supported by the National Science Foundation under the awards DMS-0908413 and DMS-1211667.} }

\author{Nicholas Michalowski
\thanks{Department of Mathematical Sciences, New Mexico State University, Las Cruces, New Mexico 88003, ({\tt nmichalo@nmsu.edu}).}
\and Stephen Pankavich
\thanks{Department of Applied Mathematics and Statistics, Colorado School of Mines, Golden, Colorado 80401 ({\tt pankavic@mines.edu}).} 
 }

\pagestyle{myheadings}
\markboth{N. MICHALOWSKI AND S. PANKAVICH}{1.5D RELATIVISTIC VLASOV-MAXWELL-FOKKER-PLANCK EQUATION}

\maketitle

\begin{abstract}
In a recent paper Calogero and Alcantara \cite{CalogeroRVMFP} derived a Lorentz-invariant Fokker-Planck equation, which corresponds to the evolution of a particle distribution associated with relativistic Brownian Motion.  We study the ``one and one-half'' dimensional version of this problem with nonlinear electromagnetic interactions - the relativistic Vlasov-Maxwell-Fokker-Planck system - and obtain the first results concerning well-posedness of solutions.  Specifically, we prove the global-in-time existence and uniqueness of classical solutions to the Cauchy problem and a gain in regularity of the distribution function in its momentum argument.  
\end{abstract}


\section{Introduction}
A plasma is a partially or completely ionized gas. Matter exists in this state if the velocities of individual particles in a material achieve magnitudes approaching the speed of light.  If a plasma is of sufficiently low density or the time scales of interest are small enough, it is deemed to be ``collisionless'', as collisions between particles become extremely infrequent.  Many examples of collisionless plasmas occur in nature, including the solar wind, the Van Allen radiations belts, and galactic nebulae.

From a mathematical perspective, the fundamental Lorentz-invariant equations which describe the time evolution of a collisionless plasma are given by the relativistic Vlasov-Maxwell system:
\begin{equation} \label{RVM} \tag{RVM} \left \{ \begin{gathered} \partial_t f + \hat{v} \cdot \nabla_x f + \left (E + \hat{v} \times
B \right ) \cdot \nabla_v f = 0 \\  \rho(t,x) = \int f(t,x,v) \ dv, \quad j(t,x)= \int \hat{v}  f(t,x,v) \ dv \\  \partial_t E = \nabla \times B - j, \qquad \nabla \cdot E = \rho \\  \partial_t B = - \nabla \times E,
\qquad \nabla \cdot B = 0. \\ \end{gathered} \right.
\end{equation}
Here, $f$ represents the distribution of (positively-charged) ions in the plasma, while $\rho$ and $j$ are the charge and current density, and $E$ and $B$ represent electric and magnetic fields generated by the charge and current.  The independent variables, $t \geq 0$ and $x,v \in \bfR^3$ represent time, position, and momentum, respectively, and physical constants, such as the charge and mass of particles, as well as, the speed of light, have been normalized to one.  The structure of the velocity terms $\hat{v}$ in (\ref{RVM}) arise due to relativistic corrections, and this quantity is defined by $$\hat{v} = \frac{v}{v_0}, \qquad v_0 = \sqrt{1 + \vert v \vert^2}.$$ 

In order to include collisions of particles with a background medium in the physical formulation, often a diffusive Fokker-Planck term is added to the Vlasov equation in (\ref{RVM}).  With this, the system is referred to as the relativistic Vlasov-Maxwell-Fokker-Planck equation.  Since basic questions of well-posedness remain unknown even in lower dimensions, we study a dimensionally-reduced version of this model for which $x \in \bfR$ and $v \in \mathbb{R}^2$, the so-called ``one and one-half dimensional'' analogue, given by
\begin{equation} \label{RVMFP} \tag{RVMFP} \left \{ \begin{gathered} \partial_t f + \hat{v}_1 \partial_x f + K \cdot \nabla_v f = \nabla_v \cdot \left ( D \nabla_v f \right ) \\  D = \frac{1}{v_0} \left [ \begin{array}{lr} 1 + v_1^2 & v_1v_2\\ v_1v_2 & 1 + v_2^2 \end{array} \right ]\\
K_1 = E_1 + \hat{v}_2 B, \qquad K_2 = E_2 - \hat{v}_1 B\\  \rho(t,x) = \int f(t,x,v) \ dv - \phi(x), \quad j(t,x)= \int \hat{v}  f(t,x,v) \ dv \\  \partial_t E_2 = -\partial_x B - j_2, \quad \partial_t B = - \partial_x E_2, \quad \partial_x E_1 = \rho, \quad \partial_t E_1 = j_1.
\end{gathered} \right.
\end{equation}
Here, we assume a single species of particles described by $f(t,x,v)$ in the presence of a given, fixed background $\phi \in C_c^1(\bfR)$ that is neutralizing in the sense that $$\int \rho(0,y) \ dy = 0.$$  The electric and magnetic fields are given by $E(t,x) = \langle E_1(t,x),E_2(t,x) \rangle$ and $B(t,x)$, respectively.  Finally,  the matrix $D = v_0^{-1} ( \mathbb{I} + v \otimes v) \in \bfR^{2 \times 2}$ is the relativistic diffusion operator and possesses some desirable properties, as discovered for its three-dimensional variant in \cite{CalogeroRVMFP}. We note, however, that the operator $\nabla_v \cdot (D \nabla_v f)$ is not uniformly elliptic and provides less dissipation than the Laplacian $\Delta_v f$.  Namely, for any $u \in \bfR^2$, $D$ satisfies
\begin{equation}
\label{D1}
v_0^{-1} \vert u \vert^2  \leq \vert u \cdot Du \vert \leq v_0 \vert u \vert^2.
\end{equation}
For initial data we take a nonnegative particle density $f^0$ with compact $x$-support and bounded moments $v_0^\b \del_x^k f^0 \in L^2(\bfR^3)$, along with fields $E^0_2, B^0 \in H^2(\bfR)$.  Additionally, we specify particular data for $E_1$, namely $$E_1(0,x) = \int_{-\infty}^x \left ( \int  f^0(y,w) \ dw - \phi(y) \right ) \ dy.$$  In fact, this particular choice of data for $E_1$ is the only one which leads to a solution possessing finite energy (see Lemma \ref{L2} and \cite{GlaSch90}). The inclusion of the neutralizing density $\phi$ is also necessary in order to arrive at finite energy solutions for (\ref{RVMFP}) with a single species of ion.

Over the past twenty-five years significant progress has been made in the analysis of (\ref{RVM}), specifically, the global existence of weak solutions (which also holds for the non-relativistic system (VM); see \cite{DPL}) and the determination of conditions which ensure global existence of classical solutions (originally discovered in \cite{GlStr}, and later in \cite{BGP} and  \cite{KlStaf} using different methods) for the Cauchy problem.  Additionally, a wide array of results have been obtained regarding electrostatic simplifications of (\ref{RVM}) - the Vlasov-Poisson and relativistic Vlasov-Poisson systems, obtained by taking the limit as $c \to \infty$ \cite{Sch86} and $B \equiv 0$, respectively.  These models do not include magnetic effects, and the electric field is given by an elliptic equation rather than a hyperbolic PDE. This simplification has led to a great deal of progress concerning the electrostatic systems, including theorems regarding the well-posedness of solutions \cite{LP, VPSSA3DGeneral, Pfaff,  SchVP}. General references on kinetic equations of plasma dynamics, such as (\ref{RVM}) and (\ref{RVMFP}), include \cite{Glassey} and \cite{VKF}.

Independent of recent advances, many of the most basic existence and regularity questions remain unsolved for (\ref{RVMFP}).   For much of the existence theory for collisionless models, one is mainly focused on bounding the velocity support of the distribution function $f$, assuming that $ f^0$ possess compact momentum support, as this condition has been shown to imply global existence \cite{GlStr}.  Hence, one of the main difficulties which arises for (\ref{RVMFP}) is the introduction of particles that are propagated with infinite momentum, stemming from the inclusion of the diffusive Fokker-Planck operator. Thus, the momentum support is necessarily unbounded and many known tools are unavailable.  Though the $v$-support of the distribution function is not bounded, we are able to overcome this issue by controlling large enough moments of the distribution to guarantee sufficient decay of $f$ in its momentum argument.  This also allows us to control the singularities which arise from representing derivatives of the fields.  As an additional difference arising from the Fokker-Planck operator, we note that when studying collisionless systems, in which $D \equiv 0$, $L^\infty$ is typically the proper space in which to estimate both the particle distribution and the fields.  With the addition of the diffusion operator, though, the natural space in which to estimate $f$ is now $L^2$.  Thus, to take advantage of the gain in regularity that should result from the Fokker-Planck term, we iterate in a weighted $L^2$ setting, estimating moments $v_0^\gamma \partial_{x,v}^k f$ in $L^2$.  Other crucial features which appear include the cone estimate, conservation of mass, and the symmetry and positivity of the diffusive operator.

Though this is the first investigation of the well-posedness of (\ref{RVMFP}), others have studied Vlasov-Maxwell models incorporating a Fokker-Planck term.  Both Yu and Yang \cite{YangYu} and Chae \cite{Chae} constructed global classical solutions to the non-relativistic Vlasov-Maxwell-Fokker-Planck system for initial data sufficiently close to Maxwellian using Kawashima estimates and the well-known energy method.  Additionally, Lai \cite{Lai, Lai2} arrived at a similar result for a one and one-half dimensional ``relativistic'' Vlasov-Maxwell-Fokker-Planck system using classical estimates. The unfortunate commonality amongst these models, however, is that they lack invariance properties. Namely, each couples the Lorentz-invariant Maxwell equations to either a Galilean-invariant Vlasov equation with non-relativistic velocities or a hybrid Vlasov equation that includes relativistic velocity corrections, but utilizes the Laplacian $\Delta_v$ as the Fokker-Planck term.  This latter term destroys the inherent Lorentz-invariance of the relativistic Vlasov-Maxwell system. Thus, we consider a diffusive operator of the form $\nabla_v \cdot (D \nabla_v f)$ which preserves this property.
With this structure in place, we can prove global existence of classical solutions under relatively relaxed assumptions:

\vspace{0.1in}

\begin{theorem}
\label{T1}
Assume the initial particle distribution satisfies $v_0^\a f^0 \in L^\infty(\bfR^3)$ and $v_0^{\b-k/2} \del_x^k f^0 \in L^2(\bfR^3)$ for some $ \a > 5,  \b > 2$, and all $k=0,1,2$.  Additionally, assume $f^0$ possesses compact support in $x$ with $E_2^0,  B^0 \in H^2(\mathbb{R})$ and $\phi \in C^1_c(\bfR)$.  Then, for any $T > 0$ there exist unique functions $f \in C^1((0,T) \times\bfR; C^2(\bfR^2)), E \in C^1((0,T) \times \mathbb{R};\bfR^2)$, and $B \in C^1((0,T) \times \mathbb{R})$ satisfying (\ref{RVMFP}) on $(0,T)$ and the Cauchy data $f(0,x,v) =  f^0(x,v)$, $E_2(0,x) = E_2^0(x)$, and $B(0,x)  =  B^0(x)$.
\end{theorem}

\vspace{0.1in}

We note that a similar global existence theorem for classical solutions can be proven by adapting the methods of Lai \cite{Lai} and Degond \cite{Degond}, which rely only on $L^\infty$ estimates of the density and its derivatives.  The initial data would need to satisfy $v_0^\a  f^0 \in C^k(\bfR^3)$ for some $\a > 3$ and $k \geq 2$ with $E_2^0,  B^0 \in C^2(\bfR)$, which is more restrictive than our assumptions and requires derivatives in $v$ initially.  Since we utilize $L^2$ estimates instead, we are able to gain derivatives in $v$ for the particle distribution.  Of course, the methods we employ are also valid in the case $D = \mathbb{I}$, and hence provide an improved global existence theorem for the systems studied by Lai, Yu-Yang, and Chae, but with less regularity imposed on the initial data.  Finally, Theorem \ref{T1} can be altered slightly to accommodate friction terms which may arise within the formulation of the model.  In this case, the Maxwell equations are unchanged and the Vlasov equation undergoes very minor alterations, taking the form
$$ \partial_t f + \hat{v}_1 \partial_x f + K \cdot \nabla_v f = \nabla_v \cdot \left ( D \nabla_v f + vf \right ).$$
The new terms are lower order and have no additional effect on the results we present. Lai has already displayed this within the context of his methods \cite{Lai2}, though the additional friction term in \cite{Lai2} is $\hat{v} f$ and not $vf$. Additionally, we note that the friction term destroys the Lorentz-invariant structure of the equation.

This paper proceeds as follows.  In the next section, we will derive \emph{a priori} estimates in order to simplify the proof of the existence and uniqueness theorem.  In Section $3$, we prove the lemmas of Section $2$, and then sketch the proof of global-in-time existence and uniqueness in Section $4$. Throughout the paper the value $C > 0$ will denote a generic constant that may change from line to line.  When necessary, we will specifically identify the quantities upon which $C$ may depend.  Regarding norms, we will abuse notation and allow the reader to differentiate certain norms via context.  For instance, $\Vert f(t) \Vert_\infty = \mathop{\esssup}\limits_{x \in \bfR,v \in \bfR^2} \vert f(t,x,v) \vert$, whereas  $\Vert B(t) \Vert_\infty = \mathop{\esssup}\limits_{x \in \bfR} \vert B(t,x) \vert$, with analogous statements for $\norm{\cdot}_2$ and $\ip{\cdot}{\cdot}$ which denote the $L^2$ norm and inner product, respectively.
Finally, for derivative estimates we will use the notation 
$$\Vert v_0^\gamma \partial_x^j \grad_v^k f(t) \Vert_2 = \sum_{\vert \alpha \vert = k} \Vert v_0^\gamma \partial_x^j \partial_v^\alpha f(t) \Vert_2$$
for $\gamma \in \bfR$, $j,k = 0,1,2,...$, and a multi-index $\alpha=(\alpha_1,\alpha_2)$ where we denote
  $\del_{v_1}^{\alpha_1}\del_{v_2}^{\alpha_2}$ by $\del_v^\alpha$.

\section{ \emph{A priori} estimates}

Let $T > 0$ be given so that we may estimate on the bounded time interval $(0,T)$ when necessary. To begin, we will first prove a result that will allow us to estimate the particle density and its moments.  When studying collisionless kinetic equations, one often wishes to integrate along the Vlasov characteristics in order to derive estimates.  However, the appearance of the Fokker-Planck term changes the structure of the operator in (\ref{RVMFP}), and the values of the distribution function are not conserved along such curves.  Hence, the following lemma (similar to that of \cite{Degond}) will be utilized to estimate the particle distribution in such situations. 
\begin{lemma}
\label{L1}
Let $g \in L^1((0,T), L^\infty(\bfR^3))$, $F \in W^{1,\infty}((0,T) \times \bfR^3; \bfR^2)$,  and $h_0 \in L^\infty(\bfR^3) \cap L^2(\bfR^3)$ be given with $D \in C(\bfR^2; \bfR^{2 \times 2})$ positive semi-definite.  Assume $h(t,x,v)$ is a weak solution of
\begin{equation}
\label{LinVlasov}
\left \{ \begin{gathered} 
\partial_t  h + \hat{v}_1 \partial_x h + F(t,x,v) \cdot \nabla_v  h - \nabla_v \cdot \left ( D \nabla_v h \right)  =  g(t,x,v)\\
  h(0,x,v) =   h_0(x,v).
\end{gathered} \right.
\end{equation}
Then, for every $t \in [0,T]$ $$\norm{h(t)}_\infty \leq \norm{h_0}_\infty + \int_0^t \norm{g(s)}_\infty \ ds.$$
\end{lemma}

Next, we state a lemma that will allow us to control the fields and moments of the particle distribution.
\begin{lemma}[Cone Estimate and Field Bounds]
\label{L2}
Assume $v_0^\a  f^0 \in L^\infty(\bfR^3)$ for some $\a > 3$, $ f^0$ possesses compact support in $x$, and $E_2^0, B^0 \in H^1(\bfR)$. Then, for any  $t \in [0,T]$, $x\in\mathbb{R}$, we have
\begin{equation}
\label{cone}
\begin{gathered}
\int_0^t \left ( \int (v_0 \pm v_1)f(s, x \pm (t-s),v) \ dv + \frac{1}{2} \vert E_1(s, x \pm (t-s)) \vert^2 \right. \\ \qquad + \left. \frac{1}{2}\vert E_2 \pm B \vert^2 (s, x \pm (t-s))   \right ) \ ds \leq C(1+t),
\end{gathered}
\end{equation}
\begin{equation}
\label{j2}
\int_0^t \vert j_2(s, x \pm (t-s)) \vert  ds \leq C(1 + t),
\end{equation}
and
\begin{equation}
\label{fieldbounds}
\Vert E(t) \Vert_\infty +\Vert B(t) \Vert_\infty \leq C(1 + t).
\end{equation}
\end{lemma}

Once control of the fields is obtained, higher moments of the particle distribution function can be controlled as well.
\begin{lemma}[Estimates on moments]
\label{L3}
Let the assumptions of Lemma \ref{L2} hold. Then, for any $\gamma \in [0,\a]$ and $t \in [0,T]$
\begin{equation}
\label{densitybounds}
\Vert v_0^\gamma f(t) \Vert_\infty \leq C(1 + t)^{2\gamma}
\end{equation}
and for any $\gamma \in [0,\a-2)$ and $t \in [0,T]$
\begin{equation}
\label{rhobounds}
\norm{ \int v_0^{\gamma} f(t) \ dv }_\infty \leq C(1 + t)^{2\a}.
\end{equation}
\end{lemma}

With control on moments of the density, we may bound derivatives of the field by adapting a well-known argument \cite{GlaSch90, GlStr} that projects these derivatives onto the backward light cone.
\begin{lemma}[Estimates on field derivatives]
\label{L4}
Let the assumptions of Lemma \ref{L2} hold, and assume additionally that $E_2^0, B^0 \in H^2(\bfR)$. Then, for any $t \in [0,T]$, we have
\begin{equation}
\label{fieldderivbounds}
\Vert \partial_t E(t) \Vert_\infty + \Vert \partial_x E(t) \Vert_\infty + \Vert \partial_t B(t) \Vert_\infty + \Vert \partial_xB(t) \Vert_\infty \leq C(1 + t)^{2(\a+1)}.
\end{equation}
\end{lemma}
Thus, we have $C^1$ estimates on the fields without requiring any regularity of the density.  Next, we utilize energy estimates to bound the density and its derivatives in $L^2(\bfR^3)$.
\begin{lemma} \label{L5} 
Assume $f^0\in L^2(\bfR^3)$.  Then, for every $t \in [0,T]$
$$ \norm{f(t)}_2 \leq \norm{ f^0}_2.$$
If additionally, $v_0^\gamma f^0\in L^2(\bfR^3)$ for some $\gamma > 0$ and the hypotheses of Lemma \ref{L2} hold, then
$$\norm{v_0^\gamma f(t)}_2 \leq C_T $$
for every $t \in [0,T]$.
\end{lemma}
\begin{lemma} \label{L6}
Assume the hypotheses of Lemma \ref{L4} hold with $v_0^{\gamma+1} f^0, v_0^\gamma \partial_x f^0 \in L^2(\bfR^3)$ for some $\gamma \geq 0$. Then for all $t \in [0,T]$ we have 
  $$\norm{v_0^\gamma \del_x f(t)}_2 \leq C_T.$$
\end{lemma}
\begin{lemma}  \label{L7}
Assume the hypotheses of Lemma \ref{L4} hold with $v_0^\a f^0 \in L^\infty(\bfR^3)$ and $v_0^{\b-k/2} \del_x^k f^0\in L^2(\bfR^3)$ for some $\a > 5$, $\b > 2$ and any $k =0,1,2$.
Then, for all $t \in [0,T]$
 $$ \Vert v_0^\gamma \partial_{xx} f(t) \Vert_2 + \sum_{k=0}^2 \left ( \Vert \partial^k E(t) \Vert_2 + \Vert \partial^k B(t) \Vert_2  \right ) \leq C_T$$
for every $\gamma \in [0, c]$,  where $c= \min \left \{\frac{\a - 3}{2}, \b-1 \right \}$ and $\partial^k$ is any $t$ or $x$ derivative of order $k$.
\end{lemma}

Next, we derive dissipative inequalities for lower-order derivatives of the density.  Ultimately, these will be used to prove the gain in regularity achieved by Lemma \ref{L8}.
\begin{lemma}[Low-order Dissipation]
\label{dissipative}
Assume the hypotheses of Lemma \ref{L7} hold.  
Then, for all $t \in (0,T)$, we have the following 
$$\frac{d}{dt} \Vert v_0^2 f(t) \Vert_2^2 \leq  C_T \Vert v_0^2 f(t) \Vert_2^2 - \Vert v_0^{3/2} \grad_v f(t) \Vert_2^2$$
$$\frac{d}{dt} \Vert v_0^{3/2} \grad_v f(t) \Vert_2^2 \leq C_T\left ( \Vert v_0^{3/2} \grad_v f(t) \Vert_2^2 + \Vert v_0 \partial_x f(t) \Vert_2^2 \right )- \Vert v_0 \grad^2_v f(t) \Vert_2^2$$
$$\frac{d}{dt} \Vert v_0^{3/2} \partial_x f(t) \Vert_2^2 \leq C_T \left (\Vert v_0^{3/2} \partial_x f(t) \Vert_2^2 + \Vert v_0^2 f(t) \Vert_2^2 \right ) - (1-\eps) \Vert v_0 \grad_v \partial_x f(t) \Vert_2^2$$
\begin{eqnarray*}
\frac{d}{dt} \Vert v_0 \grad_v \partial_x f(t) \Vert_2^2 & \leq & C_T \left ( \Vert v_0 \grad_v \partial_x f(t) \Vert_2^2 + \Vert \partial_{xx} f(t) \Vert_2^2 +  \Vert v_0^{3/2} \grad_v f(t) \Vert_2^2  \right ) \\
& \ & \ - (1-\eps) \Vert v_0^{1/2} \grad^2_v \partial_x f(t) \Vert_2^2.
\end{eqnarray*}
\end{lemma}

The next lemma contains dissipative inequalities for higher-order derivatives of the density.  In particular, it will allow us to trade $v$-derivatives of the density for those which are two orders less with the associated penalty of an $x$-derivative and a $v_0$ moment.  For instance, use of this lemma will allow us to conclude the estimates
$$\Vert \grad^4_v f(t) \Vert_2^2 \lesssim \Vert v_0^{1/2} \grad^2_v \partial_x f(t) \Vert_2^2 \lesssim \Vert \partial_{xx} f(t) \Vert_2^2 \leq C_T$$
along with the previously obtained bound on the second spatial derivative.

\begin{lemma}[High-order dissipation]
\label{v_deriv_lemma}
Assume the hypotheses of Lemma \ref{L7} hold.  
Then, for all $t \in (0,T)$, we have
\begin{eqnarray*} 
\frac{d}{dt} \norm{v_0^\gamma \grad_v^k f (t)}_2^2 & \leq &
  C_T\biggl(\norm{v_0^\gamma \grad_v^k f (t)}_2^2 +
  \sum_{j=1}^{k-1} \norm{v_0^{\gamma+j-k} \grad_v^j f (t)}_2^2
  + \norm{v_0^{\gamma+1/2}\del_x\grad_v^{k-2} f(t)}_2^2\Big) \\
  & \ & \   - (1- \eps) \norm{v_0^{\gamma -1/2} \grad_v^{k+1} f (t)}_2^2 
\end{eqnarray*}
for every $\gamma \in [0, \b -k/2]$, $k=2, 3, 4$, and $\eps > 0$ sufficiently small.
Additionally, we have 
\begin{eqnarray*}
\frac{d}{dt} \Vert v_0^{1/2} \grad_v^2 \partial_x f(t) \Vert_2^2 & \leq & C_T \biggl(\Vert v_0^{1/2} \grad_v^2 \partial_x f(t) \Vert_2^2
+ \Vert v_0 \grad_v \partial_x f(t) \Vert_2^2
+  \sum_{j=1}^2 \norm{v_0^{2 - j/2} \grad_v^j f (t)}_2^2 \\
& \ &  + \Vert \partial_{xx} f(t) \Vert_2^2  \biggr)- (1- \eps) \norm{\grad_v^3 \partial_x f (t)}_2^2 
\end{eqnarray*}
for all $t \in (0,T)$ and $\eps > 0$ sufficiently small.
\end{lemma}

%

Our final lemma removes the need for regularity of the initial density in $v$ in order to obtain derivative bounds.  Hence, solutions achieve a gain in regularity where $f$ and $\partial_x f$ are smooth in $v$ even for initial data which are not.
\begin{lemma}  \label{L8}
Assume the hypotheses of Lemma \ref{L7} hold.
Then for all $t \in (0,T)$,
$$\sum_{k=0}^4 \frac{t^k}{2^k k!} \norm{v_0^{(4-k)/2} \grad_v^k f(t)}_2^2 + \sum_{k=0}^2 \frac{t^k}{2^k k!} \norm{v_0^{(3-k)/2} \grad_v^k \partial_x f(t)}_2^2 \leq C_T.$$ 
    
\end{lemma}

The gain in regularity achieved from the momentum argument is generally expected from the diffusive term.  Additionally, it is possible that the solution gains regularity in its spatial argument as well, but this feature of the system remains unknown.  Precedent exists for this possibility, however, as analogous work of Herau \cite{Herau} and Villani \cite{Villani} has determined that this does, in fact, occur for the linear, non-relativistic Fokker-Planck equation, as long as the given potential is sufficiently smooth.

\section{Proofs of Lemmas and Estimates}

We first prove Lemma \ref{L1}, and this will require an additional result regarding the positivity of solutions to the linear Fokker-Planck equations arising from positive initial data.
\begin{proof}[Lemma \ref{L1}]
Much of our argument is adapted from ideas of Lions  \cite{LionsJL}, Tartar \cite{Tartar}, and Degond \cite{Degond}.  Thus, we sketch the proof of the lemma using results from these papers while correcting for the differences in the systems, including changes in dimension and the appearance of a diffusion operator with variable coefficients.  Consider the linear equation (\ref{LinVlasov}) and define $$\mcL h:= \partial_t h + \hat{v}_1 \partial_x h + F \cdot \nabla_v h - \nabla\cdot(D \nabla_v h).$$
We first comment that solutions $h \in L^\infty((0,T); L^2(\mathbb{R}^3))$ of the equation $\mcL h = g$ exist for any $T > 0$ and $g \in L^\infty((0,T); L^\infty(\mathbb{R}^3))$, and this follows directly from either a variational argument \cite{Degond}, the use of Green's functions \cite{VictoryODwyer}, or by properties of the heat equation on a Riemannian manifold \cite{CalogeroRVMFP}.
With this, we prove a positivity result:
\begin{lemma}
\label{Lnew}
Let $T > 0$ be given. Assume $h_0 \in L^2(\bfR^3)$ and $g \in L^\infty((0,T); L^2(\bfR^3))$ are given with $h\in L^\infty((0,T); L^2(\bfR^3))$ satisfying $\mcL h = g \geq 0$ and $h(0,x,v) = h_0(x,v) \geq 0$. Then, $h(t,x,v) \geq 0$ for all $t \geq 0, x \in \bfR, v \in \bfR^2$.
\end{lemma}

\begin{proof}[Lemma \ref{Lnew}] Let $\lambda > \frac{1}{2}\norm{\nabla_v \cdot F}_\infty$ be given. Define $u(t,x,v) = e^{-\lambda t} h(t,x,v)$ and $f(t,x,v) = e^{-\lambda t}  g(t,x,v)$.  These functions then satisfy
\begin{equation}
\label{lambda} 
\left \{ \begin{gathered}
\mcL   u + \lambda   u =  f\\
  u(0,x,v) =  h_0(x,v)
\end{gathered} \right.
\end{equation}
Let $u_-(t,x,v)= \max \{-(  u(t,x,v)),0 \}$.  In what follows, we will use the notation $\ip{\cdot }{\cdot }$ to denote the $L^2$ inner product in $(t,x,v)$ and $\norm{\cdot}_2$ to denote the corresponding induced norm.  It follows immediately from \cite{Degond, Tartar} that
\begin{equation}
\label{Tartar}
\ip{\frac{\del   u}{\del t} + \hat{v}_1 \del_x   u}{  \ u_-}= \frac{1}{2} \left( \iint \abs{  u_-(0,x,v)}^2 \, dx \, dv - \iint \abs{  u_-(t,x,v)}^2 \, dx \, dv \right) .
\end{equation}
Using this, we find
\begin{eqnarray*}
\ip{ f}{  u_-} & = & \ip{\mcL   u + \lambda   u}{  u_-}\\
& = & \ip{\frac{\del   u}{\del t} + \hat{v}_1 \del_x   u}{  u_-} 
+ \ip{F \cdot \nabla_v  u}{  u_-} - \ip{\grad_v \cdot (D \grad_v \cdot u)}{  u_-} 
+ \lambda \ip{  u}{  u_-}  
\end{eqnarray*}
For the last term we split the integral into two portions, namely
\begin{align*} 
\ip{  u}{  u_-} &= \int_0^T\iint u(t,x,v) u_{-}(t,x,v)\,dv\,dx\,dt\\ & = \int_{A} u(t,x,v) u_{-}(t,x,v)\,dv\,dx\,dt + \int_{A^c} u(t,x,v) u_{-}(t,x,v)\,dv\,dx\,dt
\end{align*}
where $A = {\{(t,x,v) : u(t,x,v) \geq 0\}}$. On the set $A$, we have $u_{-}(t,x,v) = 0$, and the corresponding integrals vanish.  On $A^c$ we have $u_{-}(t,x,v) = -u(t,x,v)$ and hence
\begin{align*}
\int_0^T \iint u(t,x,v)u_{-}(t,x,v) \,dv\,dx\,dt & =  -\int_{A^c}  \abs{u_{-}(t,x,v)}^2 \,dv\,dx\,dt\\
& =  - \int_0^T\iint \abs{u_{-}(t,x,v)}^2 \,dv\,dx\,dt.
\end{align*}
Hence, we find
$$\lambda\ip{ u}{ u_-} = -\lambda \Vert u \Vert_2$$
After a similar analysis for the other terms above, we find
$$\ip{F \cdot \nabla_v  u}{  u_-} = -\ip{F \cdot \nabla_v  u_-}{  u_-}.$$
For the diffusion term, we proceed similarly and integrate by parts to find
\begin{align*}
- \ip{\grad_v\cdot (D \grad_v \cdot u)}{  u_-}
&= \ip{D \grad_v u }{\grad_v u_{-}}\\
&= -\ip{D\grad_v u_{-}}{\grad_v   u_-}\\
&\leq  0,
\end{align*}
since $D$ is positive semi-definite. Therefore, using these identities with (\ref{Tartar}) we have the inequality
$$\ip{ f}{  u_-} \leq \frac{1}{2} \left ( \iint \abs{  u_-(0,x,v)}^2 \,dx\,dv - \iint \abs{  u_-(t,x,v)}^2 \,dx\,dv \right ) - \ip{F \cdot \nabla_v  u_-}{  u_-}  - \lambda \norm{  u_-}_2^2$$
By assumption, $ h_0(x,v) \geq 0$ and thus $  u_-(0,x,v) = 0$.  The first term above is then nonpositive and
$$ \ip{ f}{  u_-}  \leq  - \ip{F \cdot \nabla_v  u_-}{  u_-}  - \lambda \norm{  u_-}_2^2.$$
Lastly, we integrate by parts to find
\begin{align*}
-\ip{F \cdot \nabla_v  u_-}{  u_-} &=  -\int_0^T \iint F(t,x,v)\cdot \nabla_v \left ( \frac{1}{2}\vert u_-\vert^2 \right ) ,dv\,dx\,dt\\
& = \frac{1}{2} \int_0^T \iint \nabla_v \cdot F(t,x,v)  \vert u_- \vert^2 dv\,dx\,dt\\
& \leq \frac{1}{2} \Vert \nabla_v \cdot F \Vert_\infty  \Vert u_- \Vert_2^2 
\end{align*}
We finally have
$$ \ip{ f}{  u_-}  \leq  \left ( \frac{1}{2}\Vert \nabla_v \cdot F \Vert_\infty - \lambda \right ) \norm{  u_-}_2^2 \leq 0.$$
However, by hypothesis $ f(t,x,v) \geq 0$ and by definition $u_-(t,x,v) \geq 0$, so $\ip{ f}{  u_-} \geq 0$.  Therefore, it must be the case that $\norm{  u_-}_2 = 0$, from which it follows that $  u_- = 0$ and hence $ h(t,x,v) \geq 0$.
\end{proof}

Now, we utilize Lemma \ref{Lnew} and a very simple argument to finish the proof of Lemma \ref{L1}.  Let $g \in L^1((0,T); L^\infty(\bfR^{3})), F \in W^{1,\infty}((0,T) \times \bfR^3;\bfR^2)$, and $h_0 \in L^2(\bfR^3) \cap L^\infty(\bfR^3)$ be given.  Assume $h \in L^\infty((0,T); L^2(\bfR^3))$ satisfies $\mcL  h =  g(t,x,v)$ in the weak sense and $ h(0,x,v) =  h_0(x,v).$  Define $$w(t,x,v) := \norm{ h_0}_\infty + \int_0^t \norm{ g(s)}_\infty \, ds \ - \  h(t,x,v). $$  Then, we have $$w(0,x,v) = \norm{ h_0}_\infty -  h_0(x,v) \geq 0$$ and 
\begin{eqnarray*}
\mcL  w & = & \norm{g(t)}_\infty - \mcL h\\
& = & \norm{g(t)}_\infty -  g(t,x,v)\\
& \geq & 0.
\end{eqnarray*}
Thus, by Lemma \ref{Lnew}, we find $ w(t,x,v) \geq 0$, by which it follows that
$$ h(t,x,v) \leq \norm{ h_0}_\infty + \int_0^t \norm{ g(s)}_\infty\,ds$$ for all $t,x,v$.  Finally, taking the supremum in $(x,v)$, the conclusion follows.
\end{proof}

\begin{proof}[Lemma \ref{L2}]
To prove the cone estimate, we begin by using conservation of mass.  Integrating the Vlasov equation over all $(x,v)$ we find
$$ \frac{d}{dt} \iint f(t,x,v) \ dv \ dx = 0.$$ Thus, using the decay of $f^0$ we find for every $t \in [0,T]$
\begin{equation}
\label{mass}
\iint f(t,x,v) \ dv \ dx = \iint  f^0(x,v) \ dv \ dx < \infty.
\end{equation}

To derive the necessary energy identities, we first rewrite the Fokker-Planck term in the Vlasov equation as
$$\nabla_v \cdot ( D \nabla_v f) = v_0^{-1} \biggl ( \partial_{v_1} ( v_1^2 \partial_{v_1} f) + \partial_{v_2} ( v_2^2 \partial_{v_2} f) + 2v_1v_2\partial_{v_1v_2} f + \Delta_v f \biggr ).$$
Then, multiplying the Vlasov equation by $v_0$ and integrating in $v$, the Fokker-Planck term becomes
\begin{eqnarray*}
\int v_0 \nabla_v \cdot ( D \nabla_v f) \ dv & = & \int \left ( \partial_{v_1} ( v_1^2 \partial_{v_1} f) + \partial_{v_2} ( v_2^2 \partial_{v_2} f) + 2v_1v_2\partial_{v_1v_2} f + \Delta_v f \right )\\
& = & \int 2 v_1 v_2 \partial_{v_1v_2} f  \ dv_1 \ dv_2 \\
& = & 2 \int f \ dv
\end{eqnarray*}
after two integrations by parts.
Hence, using the divergence structure of the Vlasov equation, we arrive at the local energy identity
\begin{equation}
\label{localenergy}
\partial_t e + \partial_x m =2 \int f(t,x,v) \ dv
\end{equation}
where $$ e(t,x) = \int v_0 f(t,x,v) \ dv + \frac{1}{2} \left ( \vert E(t,x) \vert^2 + \vert B(t,x) \vert^2 \right )$$ and
$$ m(t,x) = \int v_1 f(t,x,v) \ dv + E_2(t,x) B(t,x).$$
Since $ f^0$ has compact support in $x$ with suitable decay in $v$, we find $v_0f^0 \in L^1(\bfR^3)$.
We can then integrate (\ref{localenergy}) over all space to deduce the global energy identity
$$\frac{d}{dt} \int e(t,x) \ dx = 2\iint f^0(x,v) \ dx \ dv$$ whence we find $$\int e(t,x) \ dx  \leq C(1+t)$$ for all $t \in [0,T)$ and $E_1, E_2, B \in L^\infty([0,T]; L^2(\bfR))$.

To derive local estimates, we fix $(t,x)$, integrate (\ref{localenergy}) along the backwards cone in space-time $\{ (s,y) \in (0,t) \times \bfR : \vert y - x \vert \leq t - s\}$, and use Green's Theorem to find
\begin{eqnarray*} \int_0^t \biggl [ (e + m)(s, x + t - s) & + & (e-m)(s, x-t+s) \biggr ] \ ds \\
& & = \int_{x-t}^{x+t} e(0,y) \ dy + 2 \int_0^t \int_{x-t+s}^{x+t-s} \int f(s, y,v) \ dv dy ds.
\end{eqnarray*}
Using the positivity of the mass and energy, the assumptions on the data, and conservation of mass, the right side satisfies
\begin{eqnarray*}
\int_{x-t}^{x+t} e(0,y) \ dy + 2\int_0^t \int_{x-t+s}^{x+t-s} \int f(s, y, v) \ dv dy ds & \leq & \int e(0,y) \ dy + 2 \int_0^t \iint f(s,y,v) \ dv dy \ ds \\
& = & \int e(0,y) \ dy + 2 \int_0^t \left ( \iint  f^0(y,v) \ dv \ dy \right ) \ ds \\
& \leq & C(1+t)
\end{eqnarray*}
and this yields the first result.

The other conclusions of the lemma then follow from the first.  More specifically, we find
\begin{equation}
\label{vineq}
v_0 \pm v_1 = \frac{v_0^2 - v_1^2}{v_0 \mp v_1} = \frac{1+ v_2^2}{v_0 \mp v_1} \geq \frac{2 \vert v_2 \vert}{v_0 \mp v_1} \geq \frac{2 \vert v_2 \vert}{2 v_0} = \vert \hat{v}_2 \vert
\end{equation}
and by (\ref{cone})
\begin{eqnarray*} \int_0^t \vert j_2(s, x \pm (t-s)) \vert ds & \leq & \int_0^t \int \vert \hat{v}_2 \vert f(s, x \pm (t-s), v) \ dv \ ds\\
&  \leq & \int_0^t \int (v_0 \pm v_1) f(s, x \pm (t-s), v) \ dv \ ds\\
& \leq & C(1 + t).
\end{eqnarray*}

Next, we represent the fields in terms of the source $j_2$ in the associated transport equations.  Either adding or subtracting the equations for $E_2$ and $B$ in (\ref{RVMFP}) yields 
$$ \partial_t (E_2 \pm B) \pm \partial_x (E_2 \pm B) = - j_2.$$
Thus, we can write the sum or difference of the fields in terms of initial data and an integral of $j_2$ along one side of the backwards cone, namely
\begin{equation}
\label{E2B}
(E_2 \pm B)(t,x) = (E_2 \pm B)(0,x \mp t) - \int_0^t j_2(s,x \mp (t-s)) \ ds.
\end{equation}
Then, in view of the previous conclusion of the lemma and the assumption on the initial fields, we find $$\Vert (E_2 \pm B)(t) \Vert_\infty \leq C(1+t)$$ and since 
$$ E_2(t,x) = \frac{1}{2} (E_2 + B)(t,x) + \frac{1}{2} (E_2 - B)(t,x),$$ and similarly for $B$, it follows that $\Vert E_2(t) \Vert_\infty$ and $\Vert B(t) \Vert_\infty$ are controlled by this same quantity.

Finally, control of $E_1$ follows from conservation of mass and the assumption on the background density.  Integrating the equation for $E_1$ and using the assumption on $E_1(0,x)$ yields $$E_1(t,x) = \int_{-\infty}^x \rho(t,y) \ dy$$ and we find for $x \in \bfR$
$$ \vert E_1(t,x) \vert \leq \iint f(t,y,v) \ dv \ dy + \Vert \phi \Vert_1 \leq C.$$  The second conclusion of the theorem then follows by adding the field estimates.
\end{proof}

\begin{proof}[Lemma \ref{L3}]
We begin by noting that $v$ is an eigenvector of $D$ since 
\begin{equation}
\label{D2}
Dv = v_0^{-1} [ \mathbb{I} + v \times v] = v_0^{-1} [ v + (v \cdot v) v ] = v_0^{-1} (1 + \vert v \vert^2) v = v_0 v.
\end{equation}
Now, let $\gamma \geq 0$ be given.  Multiplying the Vlasov equation by $v_0^\gamma$, we find
\begin{equation}
\label{Vlasovmoment}
\partial_t(v_0^\gamma f) + \partial_x(\hat{v}_1 v_0^\gamma f) + \nabla_v \cdot [ K v_0^\gamma f] - \nabla_v(v_0^\gamma) \cdot K f = v_0^\gamma \nabla_v \cdot [D \nabla_v f].
\end{equation}
We first compute the right side of this equation. Using (\ref{D1}) and  (\ref{D2}), we find
\begin{eqnarray*}
v_0^\gamma \nabla_v \cdot [D \nabla_v f] & = & \nabla_v \cdot [v_0^\gamma D\nabla_v f ] - \nabla_v (v_0^\gamma) \cdot D\nabla_v f\\
& = & \nabla_v \cdot [D\nabla_v (v_0^\gamma f) ]  - \nabla_v \cdot [D\nabla_v (v_0^\gamma) f ] - \nabla_v (v_0^\gamma) \cdot D\nabla_v f\\
& = & \nabla_v \cdot [D\nabla_v (v_0^\gamma f) ]  - \gamma \nabla_v \cdot [ v_0^{\gamma-2}f D v  ] - \gamma v_0^{\gamma-2} v \cdot D\nabla_v f\\
& = & \nabla_v \cdot [D\nabla_v (v_0^\gamma f) ]  - \gamma \nabla_v \cdot [ v_0^{\gamma-1} f v  ] - \gamma v_0^{\gamma-2} Dv \cdot \nabla_v f\\
& = & \nabla_v \cdot [D\nabla_v (v_0^\gamma f) ]  - \gamma [(\gamma-1) v_0^{\gamma - 3} \vert v \vert^2 + v_0^{\gamma-1}v \cdot \nabla_v f + 2 v_0^{\gamma-1} f  ]  - \gamma v_0^{\gamma-1} v \cdot \nabla_v f\\
& = & \nabla_v \cdot [D\nabla_v (v_0^\gamma f) ]  - \gamma (\gamma-1) v_0^{\gamma - 3} \vert v \vert^2 - 2\gamma v_0^{\gamma-1}v \cdot \nabla_v f - 2\gamma v_0^{\gamma-1} f.
\end{eqnarray*}
The next to last term here can be rewritten as
\begin{eqnarray*}
-2\gamma v_0^{\gamma-1}v \cdot \nabla_v f & = &- 2\gamma v_0^{-1} v \cdot \nabla_v (v_0^\gamma f) +  2\gamma^2 v_0^{\gamma-3} v \cdot v f\\
& = & - 2\gamma \hat{v} \cdot \nabla_v (v_0^\gamma f) +  2\gamma^2 v_0^{\gamma-1} \left ( \frac{\vert v \vert^2}{1 + \vert v \vert^2} \right )  f.
\end{eqnarray*}
Combining this with (\ref{Vlasovmoment}) yields
\begin{equation}
\label{Vgamma}
\begin{gathered}
\partial_t(v_0^\gamma f) + \partial_x(\hat{v}_1 v_0^\gamma f) + \nabla_v \cdot [ K v_0^\gamma f] - \nabla_v(v_0^\gamma) \cdot K f\\
= \nabla_v \cdot [D\nabla_v (v_0^\gamma f) ]  - \gamma (\gamma-1) v_0^{\gamma - 3} \vert v \vert^2 - 2\gamma \hat{v} \cdot \nabla_v (v_0^\gamma f) +  2\gamma^2 v_0^{\gamma-1} \left ( \frac{\vert v \vert^2}{1 + \vert v \vert^2} \right )  f   - 2\gamma v_0^{\gamma-1} f.
\end{gathered}
\end{equation}
Thus, if we rearrange terms and use the operator $$\mcV h:= \partial_t h + \hat{v}_1 \partial_x h + (K + 2\gamma \hat{v}) \cdot \nabla_v h - \nabla\cdot(D \nabla_v h),$$
we have
\begin{equation}
\label{V}
\mcV(v_0^\gamma f) = g(t,x,v)
\end{equation}
where
$$g(t,x,v) = \nabla_v(v_0^\gamma) \cdot K f - \gamma (\gamma-1) v_0^{\gamma - 3} \vert v \vert^2 + 2\gamma^2 v_0^{\gamma-1} \left ( \frac{\vert v \vert^2}{1 + \vert v \vert^2} \right )  f- 2\gamma v_0^{\gamma-1} f.$$
Estimating $g$, we find
\begin{eqnarray*}
\vert g(t,x,v) \vert & \leq & \gamma  v_0^{\gamma - 1} \vert \hat{v} \cdot K \vert f + \gamma(\gamma-1) v_0^{\gamma-1} \vert \hat{v} \vert^2 + 2\gamma^2 v_0^{\gamma-1} f + 2\gamma v_0^{\gamma -1} f\\
& \leq & \gamma  v_0^{\gamma - 1} ( \Vert E(t) \Vert_\infty + \Vert B(t) \Vert_\infty) f +  C v_0^{\gamma-1}f\\
& \leq & C(1+t) \Vert v_0^{\gamma-1}f(t) \Vert_\infty
\end{eqnarray*}

Since the coefficients of $\mcV$ satisfy the hypotheses of Lemma \ref{L1}, we use this result with $h = v_0^\gamma f$, $\mcL = \mcV$, and $g$ defined as above.  This yields
$$\Vert v_0^\gamma f(t) \Vert_\infty \leq \Vert v_0^\gamma  f^0 \Vert_\infty +C \int_0^t (1+s) \Vert v_0^{\gamma -1} f(s) \Vert_\infty \ ds.$$
Of course, the same lemma can be invoked with $h = f$ and $g = 0$ using the Vlasov equation in order to find $$\Vert f(t) \Vert_\infty \leq \Vert  f^0 \Vert_\infty$$ for all $t \in [0,T]$. With this bound on the particle distribution, which represents the $\gamma = 0$ case above, we use induction to bound $\Vert v_0^\gamma f(t) \Vert_\infty$ for any $\gamma \geq 0$ such that $\Vert v_0^\gamma  f^0 \Vert_\infty$ is finite, and the first conclusion follows.

The second conclusion is a straightforward application of the first. Namely, for any $\gamma \in [0,\a-2)$,
$$ \int v_0^\gamma f(t,x,v) \ dv \leq \norm{ v_0^\a f(t) }_\infty \left ( \int v_0^{\gamma - \a} \ dv \right ) \leq C(1+ t)^{2\a}$$
since $\gamma - \a < -2$.
\end{proof}

\begin{proof}[Lemma \ref{L4}]
We begin by noting that $E_1$ can be handled separately from the other field terms, since by Lemma \ref{L3}
$$\partial_x E_1 =  \int f(t,x,v) \ dv + \phi(x) \leq \left \Vert \int f(t) dv  \right \Vert_\infty + \Vert \phi \Vert_\infty  \leq C(1 + t)^{2\a}.$$
The same bound holds using this argument for $ \partial_t E_1 = j_1$ since $\vert \hat{v}_1 \vert \leq 1$.

Next, we represent the field equations for $E_2$ and $B$ as in the proof of Lemma \ref{L2}.  We will consider only $x$-derivatives and the term $(E_2+B)(t,x)$, but note that the same computations below can be done for $(E_2-B)(t,x)$ and time derivatives. Using (\ref{E2B}) and differentiating in $x$, we find
$$\partial_x (E_2 + B)(t,x) = (E_2 + B)'(0,x-t) - \int_0^t \int \hat{v}_2 \partial_x f(s,x - (t-s),v) \ dv ds.$$ At this point, we wish to project $\partial_x$ onto the directions of ``good'' derivatives included in the field representation.  This idea was used by Glassey and Schaeffer \cite{GlaSch90} for the collisionless problem and originally developed for the three-dimensional relativistic Vlasov-Maxwell system by Glassey and Strauss \cite{GlStr}.  We introduce the operators $$ \left \{ \begin{gathered} T_+ = \partial_t + \partial_x \\ S = \partial_t + \hat{v}_1 \partial_x \end{gathered} \right. $$ and transform $x$-derivatives on the density as $$\partial_x = \frac{1}{1-\hat{v}_1} (T_+ - S).$$ Contrastingly, the operator $T_- =  \partial_t - \partial_x$ would be needed for an estimate of $E_2 - B$.  Using the Vlasov equation, we can write $$Sf = \partial_t f + \hat{v}_1 \partial_x f  = -\nabla_v (Kf) + \nabla_v \cdot (D \nabla_v f)$$ so that integrating by parts yields
\begin{eqnarray*}
\int_0^t \int \hat{v}_2 \partial_x f(s,x-t+s,v) \ dv ds & = & \int_0^t \int \frac{\hat{v}_2}{1-\hat{v}_1} (T_+f - Sf)(s,x-t+s,v) \ dv ds\\
& = & \int_0^t \int \frac{\hat{v}_2}{1-\hat{v}_1} \biggl [ \frac{d}{ds} (f(s,x-t+s,v)) + \nabla_v \cdot (Kf)(s,x-t+s,v) \\
& \ & \quad - \nabla_v \cdot (D\nabla_v f) (s,x-t+s,v) \biggr ] \ dv ds\\
& = & \int \frac{\hat{v}_2}{1-\hat{v}_1} [f(t,x,v) -  f^0(x-t,v)] \ dv \\
& \ & \quad + \int_0^t \int  \frac{\hat{v}_2}{1-\hat{v}_1}  \nabla_v \cdot (Kf)(s,x-t+s,v) \ dv ds\\
& \ & \quad - \int_0^t \int \frac{\hat{v}_2}{1-\hat{v}_1}  \nabla_v  \cdot (D\nabla_v f) (s,x-t+s,v) \biggr ] \ dv ds\\
& =: & I + I + III
\end{eqnarray*}

The first term is easily estimated since moments of the density are bounded.  We use (\ref{vineq}) and $\a > 3$ to find
\begin{eqnarray*}
I & = & \int \frac{v_2}{v_0-v_1} [ f(t,x,v) -  f^0(x-t,v)] \ dv \\
& \leq & \int \frac{\vert v_2 \vert(v_0+v_1)}{1+v_2^2} f(t,x,v) \ dv\\
& \leq & \Vert v_0^\a f(t) \Vert_\infty \int v_0^{1-\a} \ dv \leq C (1 + t)^{2\a}.
\end{eqnarray*}

To estimate $II$, we first integrate by parts to find
\begin{eqnarray*}
\int_0^t \int  \frac{\hat{v}_2}{1-\hat{v}_1}  \nabla_v \cdot (Kf)(s,x-t+s,v) \ dv ds = - \int_0^t \int \nabla_v  \left ( \frac{\hat{v}_2}{1-\hat{v}_1} \right ) \cdot (Kf)(s,x-t+s,v) \ dv ds\\
+ \lim_{\vert v \vert \to \infty} \int_0^t \frac{\hat{v}_2}{1-\hat{v}_1} (Kf)(s,x-t+s,v) \cdot \frac{v_\perp}{\vert v \vert}\ ds 
\end{eqnarray*}
where $v_\perp = \langle v_2, -v_1 \rangle.$
The boundary term vanishes on $(0,T)$ because $K$ and $v_0^\a f$ are bounded in $L^\infty$, and thus moments can be used to introduce sufficient decay in $v$.  For the remaining term, we compute the gradient
$$ \nabla_v  \left ( \frac{\hat{v}_2}{1-\hat{v}_1} \right ) =  \left \langle \frac{\hat{v}_2}{v_0-v_1}, \frac{1}{v_0-v_1} - \frac{\hat{v}_2v_2}{(v_0-v_1)^2} \right \rangle$$
The first term is bounded since (\ref{vineq}) implies $\vert \hat{v}_2 \vert \leq v_0 - v_1$. Similarly, one can show the second term is bounded by $3v_0$ using (\ref{vineq}).  Hence, using $\a > 3$, we have
$$II \leq \int_0^t \Vert K(s) \Vert_\infty \Vert v_0^\a f(s) \Vert_\infty  \left ( \int v_0^{1-\a} \ dv \right ) \ ds  \leq C(1+t)^{2(\a + 1)}$$

Finally, we use the symmetry of $D$ and integrate by parts twice in $III$ to find 
\begin{eqnarray*}
\int_0^t \int \frac{\hat{v}_2}{1-\hat{v}_1}  \nabla_v  \cdot (D\nabla_v f) \biggr |_{(s,x-t+s,v)} \ dv ds & = & \int_0^t \int \nabla_v \cdot \left [ D\nabla_v  \left ( \frac{\hat{v}_2}{1-\hat{v}_1} \right ) \right ] f(s,x-t+s,v) \ dv ds\\
& \ & \ + \lim_{\vert v \vert \to \infty} \int_0^t \frac{\hat{v}_2}{1-\hat{v}_1} \nabla_v f(s,x-t+s,v) \cdot D v_\perp \frac{1}{\vert v \vert} \ ds \\
& \ & \ - \lim_{\vert v \vert \to \infty} \int_0^t \nabla_v \left (\frac{\hat{v}_2}{1-\hat{v}_1} \right ) \cdot Dv_\perp\frac{1}{\vert v \vert} f(s,x-t+s),v) \ ds 
\end{eqnarray*}
For the boundary terms, we use the property $D v_\perp = v_0^{-1} v_\perp$ so that an extra order of decay appears, and these terms vanish on $(0,T)$. 
To estimate the remaining term, a long computation yields the bound
$$\biggl \vert \nabla_v \cdot \left [ D\nabla_v  \left ( \frac{\hat{v}_2}{1-\hat{v}_1} \right ) \right ] \biggr \vert \leq 4.$$ Thus, we find for $\a > 2$
$$ III \leq 4 \int_0^t \int f(s,x-t+s,v) \ dv ds \leq \left ( \int _0^t \Vert v_0^\a f(s) \Vert_\infty ds \right )  \left ( \int v_0^{-\a} dv \right )\leq C(1 + t)^{2\a+1}.$$
Combining the estimates and using the regularity of the initial fields, each term is controlled by $C(1+t)^{2(\a +1)}$.  Thus, the bound on $\Vert \partial_x(E_2 + B)(t) \Vert_\infty$ follows,  as does the conclusion of the lemma. 
\end{proof}

\begin{proof}[Lemma \ref{L5}]
  We proceed by using energy estimates.  We calculate:
  \begin{align*}
  \frac{1}{2}\frac{d}{dt}\norm{f(t)}_2^2 &= \ip{-\hat{v}_1\del_x f-K\cdot \grad_v f+\grad_v\cdot D\grad_v f}{f}\\  
  &= -\ip{\hat{v}_1\del_x f}{f}-\ip{K\cdot \grad_vf}{f}+\ip{\grad_v\cdot D\grad_v f}{f}.
\end{align*}
  Notice that the first two terms are pure derivatives in $x$ and $v$, respectively.  Thus,
  $$\ip{\hat{v}_1\del_x f}{f}= \frac{1}{2} \iint \partial_x \left ( \hat{v}_1 f^2 \right ) \ dv \ dx = 0$$
  and
  $$\ip{K\cdot \grad_v f}{f}= \frac{1}{2}  \iint \nabla_v \cdot \left ( K f^2 \right ) \ dv \ dx = 0 .$$   
Finally $\ip{\grad_v\cdot D\grad_v
    f}{f}=-\norm{D^{1/2}\grad_vf(t)}_2$.  Hence
  $\frac{d}{dt}{\norm{f(t)}_2^2}\leq 0$ and the first conclusion follows.  Similarly, we may multiply by $v_0^{2\gamma}$ and proceed in the same manner
\begin{align*}
 \frac{1}{2}\frac{d}{dt}\norm{v_0^\gamma f(t)}_2^2 &= 
   -\ip{v_0^\gamma \hat{v}_1\del_x f}{v_0^\gamma f}
   -\ip{v_0^\gamma K\cdot \grad_v f}{v_0^\gamma f}
   +\ip{v_0^\gamma \grad_v\cdot D\grad_v f}{ v_0^\gamma f}
\end{align*}
   As in the previous conclusion of the lemma the first term is zero.
   Integrating by parts in the second term we find
   \begin{align*}
   -\ip{v_0^\gamma K\cdot \grad_v f}{v_0^\gamma f} & = \iint v_0^{2\gamma} K \cdot \nabla_v \left ( \frac{1}{2} f^2 \right ) \ dv \ dx \\
   &= - 2\gamma \iint v_0^{2\gamma-1} (\hat{v} \cdot K) \frac{1}{2} f^2 \ dv \ dx 
   \end{align*}
 Hence, this yields
   $$ \left \vert \ip{v_0^\gamma K\cdot \grad_v f}{v_0^\gamma f}  \right \vert \leq
C\norm{K(t)}_\infty\norm{v_0^{\gamma-\frac{1}{2}}f(t)}_2.$$
For the last term we integrate by parts and use the symmetry of $D$,
   \begin{align*}
   \ip{v_0^\gamma \grad_v\cdot (D\grad_vf)}{v_0^\gamma f}  &=
   -2\gamma\ip{v_0^{2\gamma-1} \hat{v}\cdot D\grad_vf}{f}-\norm{v_0^\gamma D^{1/2}\grad_v f(t)}_2^2\\
   &= -2\gamma\ip{v_0^{2\gamma-1}v \cdot \grad_vf}{f} -\norm{v_0^\gamma D^{1/2}\grad_v f(t)}_2^2\\
   &= -2\gamma \iint v_0^{2\gamma-1} v\cdot \grad_v \left ( \frac{1}{2} f^2 \right)  \ dv \ dx -\norm{v_0^\gamma D^{1/2}\grad_v f(t)}_2^2
   \end{align*}
 We may drop the latter term.  After integrating by parts again in
   $v$ we can bound the former term by $C\norm{v_0^{\gamma-\frac{1}{2}} f(t)}_2^2$.
   Putting the estimates together and using the field bound of Lemma \ref{L2}, we find 
   $$\frac{1}{2}\frac{d}{dt}\norm{v_0^\gamma f(t)}_2^2 \leq C(1+t) \norm{v_0^{\gamma-\frac{1}{2}} f(t)}_2^2.$$
   Using the first conclusion of the lemma for the $\gamma = 1/2$ case and proceeding by induction yields $$\norm{v_0^\gamma f(t)}_2^2
   \leq C(1+t)^{4\gamma}\norm{v_0^\gamma  f^0}_2^2 \leq C_T.$$ for every $\gamma \geq 0$ for which the norm of the initial data is finite.
\end{proof}

\begin{proof}[Lemma \ref{L6}]
To begin, we estimate derivatives of the density in $x$, and first define some notation. Since density derivatives will depend upon field derivatives, we let 
$$\mathcal{F}(t) =  \Vert E(t) \Vert_\infty + \Vert B(t) \Vert_\infty +\Vert \partial_x E(t) \Vert_\infty + \Vert \partial_x B(t) \Vert_\infty$$ and note that $\Vert \mathcal{F} \Vert_\infty \leq C_T$ by Lemmas \ref{L2} and \ref{L4}. We differentiate the Vlasov equation in $x$, multiply by $v_0^{2\gamma} \partial_x f$ and integrate to yield
\begin{eqnarray*}
\frac{1}{2} \frac{d}{dt} \left \Vert v_0^\gamma \partial_x f (t)\right \Vert_2^2 & = & - \iint \partial_x \left ( \frac{1}{2} \hat{v}_1 v_0^{2\gamma}  \vert \partial_{x} f \vert^2 \right )  dv \ dx - \iint v_0^{2\gamma} \partial_x f \nabla_v \cdot \left ( \partial_x K f + K \partial_x f \right ) \ dx \ dv\\
& \ &     + \iint v_0^{2\gamma} \partial_x f \nabla_v \cdot ( D \nabla_v \partial_x f) \ dv \ dx\\
& = & \iint \left [ v_0^{2\gamma-1} \left ( 2\gamma \hat{v} \partial_x f + v_0 \nabla_v \partial_x f \right ) \cdot ( \partial_x K f) + \frac{1}{2}v_0^{2\gamma} K\cdot \nabla_v (\vert \partial_x f \vert^2) \right ] \ dx \ dv \\
& \ &    - \iint v_0^{2\gamma-1} \left (2\gamma \hat{v}  \partial_x f + v_0  \nabla_v \partial_x f \right ) \cdot ( D \nabla_v \partial_x f) \ dv \ dx\\
& =: & I + II
\end{eqnarray*}
Here, we have integrated by parts in $v$ and used the divergence-free structure of $K$, as well as, the fact that the transport term above is a pure $x$-derivative along with the compact $x$-support of the particle distribution.  Using Cauchy's inequality with $\eps$ we find for any $\eps > 0$
\begin{eqnarray*}
I & \leq &C\iint \left [ \Vert \partial_x K(t) \Vert_\infty\left ( v_0^{2\gamma-1} \vert \partial_x f \vert f + v_0^{2\gamma} \vert \nabla_v \partial_x f \vert f \right ) + \Vert K(t) \Vert_\infty v_0^{2\gamma-1} \vert \partial_x f \vert^2 \right ] dx dv \\
& \leq & C\mathcal{F}(t) \left ( \norm{ v_0^{\gamma -\frac{1}{2}} \partial_x f(t) }_2^2 + \norm{ v_0^{\gamma-\frac{1}{2}}f(t) }_2^2 +  {\eps}\norm{ v_0^{\gamma -\frac{1}{2}} \nabla_v \partial_x f(t) }_2^2 + \frac{1}{\eps}\norm{ v_0^{\gamma+\frac{1}{2}} f (t)}_2^2 \right )\\
& \leq & C_T \left ( \norm{ v_0^{\gamma -\frac{1}{2}} \partial_x f(t) }_2^2 + \norm{ v_0^{\gamma+\frac{1}{2}} f(t) }_2^2+ \eps \norm{ v_0^{\gamma -\frac{1}{2}} \nabla_v \partial_x f(t) }_2^2 \right ) 
\end{eqnarray*}

Then, the symmetry of $D$ along with $D \hat{v} = v$ implies
\begin{eqnarray*}
II & = & - \iint v_0^{2\gamma-1} \left ( 2\gamma \hat{v}  \partial_x f + v_0  \nabla_v \partial_x f \right ) \cdot ( D \nabla_v \partial_x f) \ dv \ dx\\
& = & -\gamma \iint v_0^{2\gamma - 1} v \cdot \nabla_v (\vert \partial_x f \vert^2) \ dv \ dx - \Vert v_0^\gamma D^{1/2} \nabla_v \partial_x f \Vert_2^2 \\
& \leq & C \norm{ v_0^{\gamma -\frac{1}{2}} \partial_x f (t)}_2^2 - \norm{ v_0^{\gamma-\frac{1}{2}} \nabla_v \partial_x f(t) }_2^2. 
\end{eqnarray*}
Combining $I$ and $II$, we use Lemma \ref{L5} to find for $\eps$ sufficiently small
$$ \frac{d}{dt} \norm{v_0^\gamma \partial_x f(t) }_2^2  \leq C_T \left ( \norm{ v_0^{\gamma -\frac{1}{2}} \partial_x f (t)}_2^2  + \norm{ v_0^{\gamma+\frac{1}{2}} f(t) }_2^2\right ) \leq C_T \left (1 +  \norm{ v_0^{\gamma -\frac{1}{2}} \partial_x f(t) }_2^2 \right).$$
If we compute this for $\gamma = 0$ and use the bound on $\mathcal{F}$, the result is just 
$$ \frac{d}{dt} \left \Vert \partial_x f(t) \right \Vert_2^2  \leq C_T \Vert f(t) \Vert_2^2 $$ which, by Lemma \ref{L5}, leads to
$$ \left \Vert \partial_x f(t) \right \Vert_2^2  \leq C_T \left ( \left \Vert \partial_x f^0 \right \Vert_2^2 + \left \Vert f^0 \right  \Vert_2^2 \right ) \leq C_T$$ for every $t \in [0,T]$.
Then, by induction, for every $\gamma \geq 0$ for which $v_0^\gamma \partial_x f \in L^2(\bfR^3)$ we have
\begin{equation}
\label{Dxf}
 \left \Vert v_0^\gamma \partial_x f(t) \right \Vert_2^2  \leq C_T\left (1 + \left \Vert v_0^\gamma \partial_x f^0 \right \Vert_2^2 \right ) \leq C_T
 \end{equation}
 for all $t \in [0,T]$.
\end{proof}

\begin{proof}[Lemma \ref{L7}]
To begin, we estimate second derivatives of the density. These involve second derivatives of the fields, which must be estimated in $L^2$ rather than $L^\infty$.  As before, denote $$\mathcal{F}(t) =  \Vert E(t) \Vert_\infty + \Vert B(t) \Vert_\infty +\Vert \partial_x E(t) \Vert_\infty + \Vert \partial_x B(t) \Vert_\infty $$ and now let $$ \mathcal{G}(t) = \Vert \partial_{xx} E(t) \Vert_2^2 + \Vert \partial_{xx} B(t) \Vert_2^2.$$ We differentiate the Vlasov equation twice in $x$, multiply by $v_0^{2\gamma} \partial_{xx} f$ and integrate to yield

\begin{eqnarray*}
\frac{1}{2} \frac{d}{dt} \left \Vert v_0^\gamma \partial_{xx} f(t) \right \Vert_2^2 & = & - \iint \partial_x \left ( \frac{1}{2} \hat{v}_1  v_0^{2\gamma}  \vert \partial_{xx} f \vert^2 \right )  dv \ dx\\
& \ &    - \iint v_0^{2\gamma} \partial_{xx} f \nabla_v \cdot \left ( \partial_{xx} K f + 2\partial_x K \partial_x f + K \partial_{xx} f \right ) \ dx \ dv\\
& \ &     + \iint v_0^{2\gamma} \partial_{xx} f \nabla_v \cdot ( D \nabla_v \partial_{xx} f) \ dv \ dx\\
& = & \iint \left [ v_0^{2\gamma-1} ( 2\gamma \hat{v} \partial_{xx} f + v_0 \nabla_v \partial_{xx} f ) \cdot ( \partial_{xx} K f + \partial_x K \partial_x f) + \frac{1}{2}v_0^{2\gamma} K\cdot \nabla_v (\vert \partial_{xx} f \vert^2)  \right ] dx dv \\
& \ &    - \iint v_0^{2\gamma-1} \left ( 2\gamma \hat{v}  \partial_{xx} f + v_0  \nabla_v \partial_{xx} f \right ) \cdot ( D \nabla_v \partial_{xx} f) \ dv \ dx\\
& = & I + II
\end{eqnarray*}
As before, we have integrated by parts in $v$ and used the compact $x$-support of the particle distribution.  With this, bounds for $I$ follow as in Lemma \ref{L6} with the exception of terms involving $\partial_{xx} K$.  More specifically, we use Lemmas \ref{L2}, \ref{L5}, \ref{L6}, Cauchy-Schwarz, and Cauchy's inequality to find
\begin{eqnarray*}
I & \leq &C \mathcal{G}(t)^{1/2} \left [\int \left ( \int (v_0^{2\gamma-1} \vert \partial_{xx} f \vert + v_0^{2\gamma} \vert \nabla_v \partial_{xx} f \vert ) \ f \ dv \right )^2 dx \right ]^{1/2} \\
& \ &    + C\mathcal{F}(t) \left ( \Vert v_0^{\gamma -\frac{1}{2}} \partial_{xx} f(t) \Vert_2^2 + \Vert \partial_x f(t) \Vert_2^2 +  \eps \Vert v_0^{\gamma -\frac{1}{2}} \nabla_v \partial_{xx} f (t)\Vert_2^2 + \frac{1}{\eps}\Vert v_0^{\gamma+\frac{1}{2}} \partial_x f (t)\Vert_2^2 \right )\\
& \leq & C \mathcal{G}(t)^{1/2} \left [\int \left ( \int v_0^{2\gamma-1} \vert \partial_{xx} f \vert^2 dv \right ) \cdot \left ( \int v_0^{2\gamma-1} f^2 dv \right) dx + \int \left (\int v_0^{2\gamma-1} \vert \nabla_v \partial_{xx} f \vert^2 dv \right ) \cdot \left ( \int v_0^{2\gamma+1} f^2 dv \right) dx \right]^{1/2} \\
& \ &    + C_T \left ( 1 + \left \Vert v_0^{\gamma -\frac{1}{2}} \partial_{xx} f(t) \right \Vert_2^2 + \frac{\eps}{2} \left \Vert v_0^{\gamma -\frac{1}{2}} \nabla_v \partial_{xx} f(t) \right \Vert_2^2 \right ) \\
& \leq & C_T \left (1 + \left (1 + \frac{1}{\eps} \right )\mathcal{G}(t) + \norm{f^0}_\infty \norm{\int v_0^{2\gamma+1} f(t)}_\infty \left[ \left \Vert v_0^{\gamma -\frac{1}{2}} \partial_{xx} f(t) \right \Vert_2^2 + \eps \left \Vert v_0^{\gamma -\frac{1}{2}} \nabla_v \partial_{xx} f(t) \right \Vert_2^2 \right ]  \right )\\ 
& \leq & C_T \left (1 + \mathcal{G}(t) + \left \Vert v_0^{\gamma -\frac{1}{2}} \partial_{xx} f(t) \right \Vert_2^2 + \eps \left \Vert v_0^{\gamma -\frac{1}{2}} \nabla_v \partial_{xx} f(t) \right \Vert_2^2  \right ) 
\end{eqnarray*}
for $2\gamma < \a - 3$.
We estimate $II$ exactly as before to find 
$$ II \leq C \left \Vert v_0^{\gamma -\frac{1}{2}} \partial_{xx} f(t) \right \Vert_2^2 - \left \Vert v_0^{\gamma-\frac{1}{2}} \nabla_v \partial_{xx} f(t) \right \Vert_2^2.$$
Hence, combining $I$ and $II$, we find for $\eps$ small enough,
$$ \frac{d}{dt} \left \Vert v_0^\gamma \partial_{xx} f(t) \right \Vert_2^2  \leq C_T \left (1 + \mathcal{G}(t) + \left \Vert v_0^{\gamma -\frac{1}{2}} \partial_{xx} f(t) \right \Vert_2^2 \right ) \leq C_T \left (1 + \mathcal{G}(t) + \left \Vert v_0^\gamma \partial_{xx} f(t) \right \Vert_2^2 \right ).$$
By Gronwall's Lemma we have 
\begin{equation}
\label{Dxxf}
 \left \Vert v_0^\gamma \partial_{xx} f(t) \right \Vert_2^2  \leq C_T\left(1 + \mathcal{G}(t) \right )
 \end{equation}
 for all $t \in [0,T]$ and $\gamma < \min \left \{\frac{\a - 3}{2}, \b-1 \right \}$.


Before turning to field derivatives, we will need a way to relate the current density and its derivatives to that of the particle distribution.  So, for $k=0, 1, 2$ we estimate
\begin{equation}
\label{jf}
\Vert \partial^k_x j_2(t) \Vert_2^2 \leq \int \left ( \int  \vert \partial^k_x f \vert dv \right )^2 \ dx \leq \left ( \iint v_0^{2\gamma} \vert \partial^k_x f\vert^2 dv dx \right) \left ( \int v_0^{-2\gamma} dv \right ) \leq C\Vert v_0^\gamma \partial^k_x f (t)\Vert_2^2.
\end{equation}
for $\gamma > 1$.
Additionally, we will need to bound $\partial_t j_2$ in $L^2$, which can be done using (\ref{Dxf}).  Using the Vlasov equation and integrating by parts in $v$, we see
\begin{eqnarray*}
\partial_t j_2 & = &  - \int \hat{v}_1 \hat{v}_2 \partial_x f + \int \nabla_v (\hat{v}_2) \cdot K f \ dv + \int \nabla_v \cdot [D \nabla_v (\hat{v}_2)] f \ dv \\
& \leq & \int \vert \partial_x f \vert dv + \int (1 + \Vert K(t) \Vert_\infty) f \ dv
\end{eqnarray*}
Thus, it follows by Lemmas \ref{L2} and \ref{L5} that
\begin{equation}
\label{Dtj}
\Vert \partial_t j_2(t) \Vert_2^2 \leq C_T \left (1 + \Vert v_0^\gamma \partial_x f(t) \Vert_2^2 \right ) \leq C_T.
\end{equation}
for every $ t\in [0,T]$ where $\gamma > 1$.

Now, we estimate field derivatives.  Since $\partial_x E_1 = \rho$, we find for all $t \in [0,T]$
\begin{eqnarray*}
\Vert \partial_x E_1(t) \Vert_2^2 & \leq & C \int \left (\int f(t,x,v) \ dv \right )^2 dx + \Vert \phi \Vert_2^2\\
& \leq & C \left (1 + \Vert v_0^\gamma f(t) \Vert_2^2 \left ( \int v_0^{-2\gamma} \ dv \right ) \right )\\
& \leq & C_T
\end{eqnarray*}
by Lemma \ref{L5} where $\gamma > 1$ to bound the integral.
We estimate identically for $\partial_{xx} E_1$ and use $\phi \in C_c^1$ so that by Lemma \ref{L6} with $\gamma > 1$
$$ \Vert \partial_{xx} E_1(t) \Vert_2^2 \leq C \left (\Vert \phi' \Vert_2 + \Vert v_0^\gamma \partial_x f(t) \Vert_2^2 \left ( \int v_0^{-2\gamma} \ dv \right ) \right ) \leq C_T.$$

Using the transport equations of (\ref{RVMFP}) for $E_2$ and $B$, it follows that these quantities and their derivatives satisfy wave equations with derivatives of $j_2$ as source terms, namely
$$ \Box B = \partial_x j_2,    \quad \Box E_2 = -\partial_t j_2.$$
Using standard $L^2$ estimates for the wave equation, we multiply the first equation by $\partial_t B$ and integrate in $x$.  After integrating by parts and using Cauchy's inequality, this yields
$$\frac{d}{dt} \left ( \Vert \partial_t B(t)\Vert_2^2 +  \Vert \partial_x B(t) \Vert_2^2 \right ) \leq \Vert \partial_x j_2(t) \Vert_2^2 + \Vert \partial_t B(t) \Vert_2^2.$$
Using Lemma \ref{L6} with (\ref{jf}), this becomes
$$\frac{d}{dt} \left ( \Vert \partial_t B(t) \Vert_2^2 +  \Vert \partial_x B(t) \Vert_2^2 \right ) \leq C_T \left( 1+ \Vert \partial_t B(t) \Vert_2^2 \right )$$
which, by Gronwall's inequality, yields $$ \Vert \partial_t B(t) \Vert_2^2 +  \Vert \partial_x B(t) \Vert_2^2 \leq C_T.$$  Since $\partial_x E_2 = - \partial_t B$ and $\partial_x E_2 = - \partial_x B - j_2$, the same bounds hold for derivatives of $E_2$.

We may now proceed in a similar fashion for second derivatives of the field.  From the field equations, we see 
$$ \Box (\partial_x B) = \partial_{xx} j_2$$ and thus
$$\frac{d}{dt} \left ( \Vert \partial_{tx} B(t) \Vert_2^2 +  \Vert \partial_{xx} B(t) \Vert_2^2 \right ) \leq \Vert \partial_{xx} j_2(t) \Vert_2^2 + \Vert \partial_{tx} B(t) \Vert_2^2.$$
Since $\partial_{tx} B = -\partial_{xx} E_2$, this is equivalent to
$$\frac{d}{dt} \left ( \Vert \partial_{xx} E_2(t) \Vert_2^2 +  \Vert \partial_{xx} B(t) \Vert_2^2 \right ) \leq \Vert \partial_{xx} j_2(t) \Vert_2^2 + \Vert \partial_{xx} E_2(t) \Vert_2^2.$$ 
Using (\ref{Dxxf}) and (\ref{jf}), this implies $$\mathcal{G}'(t) \leq C_T(1 + \mathcal{G}(t) ),$$ and using Gronwall's inequality and the assumption on the initial fields, we find
$$ \Vert \partial_{xx} E_2(t) \Vert_2^2 +  \Vert \partial_{xx} B(t) \Vert_2^2  \leq C_T.$$
With this, (\ref{Dxxf}) provides an a priori bound on $\Vert v_0^\gamma \partial_{xx} f(t) \Vert_2^2$ for all $t \in [0,T]$.
Since $\Box B = -\partial_x j_2$, we see that $\partial_{tt} B = \partial_{xx}B - \partial_x j_2 \in L^\infty ([0,T]; L^2(\bfR))$ by (\ref{Dxf}) and (\ref{jf}).  Then, $\partial_{tx} E_2 = - \partial_{tt} B \in L^\infty ([0,T]; L^2(\bfR))$ and $\partial_{tx} B = - \partial_{xx} E_2 \in L^\infty ([0,T]; L^2(\bfR))$, and finally $\partial_{tt} E_2 = - \partial_{tx} B - \partial_t j_2 \in L^\infty ([0,T]; L^2(\bfR))$ by (\ref{Dtj}).
\end{proof}

\begin{proof}[Lemma \ref{dissipative}]
Throughout, we will use $v_0 \geq 1$ in order to increase moments of the estimates where necessary so as to match the results of the lemma. 
Additionally, we will use the notation $R^\gamma(v)$ to generically denote a function of $v$ such that 
$|R^\gamma(v)|\leq C_T v_0^\gamma$, but the specific value of $R^\gamma(v)$ may change from line to line.
We first estimate moments of the density.  Computing
\begin{eqnarray*}
\frac{1}{2}\frac{d}{dt} \Vert v_0^2 f(t) \Vert_2^2 & = & \iint v_0^4 f \left [-\hat{v}_1 \partial_x f - K \cdot \grad_v f + \grad_v \cdot (D \grad_v f ) \right ] \ dv dx\\
& = & I + II + III.
\end{eqnarray*}
The first term vanishes as it is a pure $x$-derivative.
For $II$, we integrate by parts and use the field bounds of Lemma \ref{L2}  so that
\begin{eqnarray*}
II & = &  - \iint v_0^4 \grad_v \cdot (K f^2) \ dv dx\\
& = & 4 \iint v_0^3 \hat{v} \cdot K f^2 \ dv dx  \\
& \leq & C_T \Vert v_0^{3/2} f(t) \Vert_2^2.
\end{eqnarray*}
To estimate $III$, we integrate by parts, then use the property $D\hat{v} = v$ and integrate by parts again in the first term.  Also, we use (\ref{D1}) in the second term to find
\begin{eqnarray*}
III & = &  - \iint \grad_v (v_0^4 f) \cdot D\grad_v  f \ dv dx\\
& = & - \iint (4v_0^3 \hat{v}f + v_0^4 \grad_v f)  \cdot D\grad_v  f \ dv dx\\
& \leq & C \Vert v_0^{3/2} f(t) \Vert_2^2 - \Vert v_0^2 D^{1/2} \grad_v f(t) \Vert_2^2.\\
& \leq & C \Vert v_0^{3/2} f(t) \Vert_2^2 - \Vert v_0^{3/2} \grad_v f(t) \Vert_2^2.
\end{eqnarray*}
Combining the estimates, the first inequality follows.

Next, we let $\partial_v$ be either first-order derivative and compute
\begin{eqnarray*}
\frac{1}{2}\frac{d}{dt} \Vert v_0^{3/2} \partial_v f(t) \Vert_2^2 & = & \iint v_0^3 \partial_v f \left [ - \hat{v}_1 \partial_v \partial_x f - \partial_v \hat{v}_1 \partial_x f \right. \\
 & \ & \ - \partial_v K \cdot \grad_v f - K \cdot \grad_v \partial_v f\\
 & \ & \left. \  + \grad_v \cdot ( (\partial_v  D) \grad_v f ) + \grad_v \cdot (D \grad_v  \partial_v  f ) \right ] \ dv dx\\
& = & I + II + III.
\end{eqnarray*}
The first term in $I$ vanishes as before and thus using Cauchy's inequality
\begin{eqnarray*}
I & = &  - \iint v_0^3 R^{-1}(v) \partial_v f \partial_x f \ dv dx\\
& \leq & C \left ( \Vert v_0 \partial_v f(t) \Vert_2^2 + \Vert v_0 \partial_x f(t) \Vert_2^2 \right).
\end{eqnarray*}
For $II$, we use the field bounds of Lemma \ref{L2} and integrate by parts in the second term to find
\begin{eqnarray*}
II & = &  - \iint v_0^3  \partial_v f [ R^{-1}(v) \grad_v f + K \cdot \grad_v \partial_v f] \ dv dx\\
& \leq & C_T \Vert v_0^{3/2} \grad_v f(t) \Vert_2^2.
\end{eqnarray*}
Finally, in $III$ we integrate by parts while using $D\hat{v} = v$ and boundedness of derivatives of $D$ to find
\begin{eqnarray*}
III & = &  - \iint \left [3v_0^2 \hat{v} \partial_v f + v_0^3 \grad_v \partial_v f \right ] [ \partial_v D \grad_v f + D \grad_v \partial_v f] \ dv dx\\
& = &  - \iint \left [3v_0^2 R^0(v) \partial_v f \grad_v f  + \frac{1}{2} v_0^3 R^0(v) \partial_v  \left \vert \grad_v f  \right \vert^2+\frac{3}{2}v_0^2 D\hat{v} \cdot  \grad_v \vert \partial_v f \vert^2 + v_0^3 \grad_v \partial_v f \cdot D \grad_v \partial_v f \right ] \ dv dx\\
& \leq & C \Vert v_0 \grad_v f(t) \Vert_2^2 - \Vert v_0 \grad_v \partial_v f(t) \Vert_2^2.
\end{eqnarray*}
We collect these estimates, use  $\Vert \partial_v f(t) \Vert_2^2 \leq \Vert \grad_v f(t) \Vert_2^2$, and then sum over first-order $v$-derivatives to arrive at an estimate on $\frac{d}{dt} \Vert v_0^{3/2} \grad_v f(t) \Vert_2^2$.  With this, the second result follows.

The final two results concern $x$-derivatives of the density, so we first compute
\begin{eqnarray*}
\frac{1}{2}\frac{d}{dt} \Vert v_0^{3/2} \partial_x f(t) \Vert_2^2 & = & \iint v_0^3 \partial_x f \left [-\hat{v}_1 \partial_{xx} f - \grad_v \cdot (\partial_x K f ) - K \cdot \grad_v \partial_x f + \grad_v \cdot (D \grad_v \partial_xf ) \right ] \ dv dx\\
& = & I + II + III + IV.
\end{eqnarray*}
As in the other estimates, $I$ vanishes.  For $II$, we integrate by parts and use the bounds on field derivatives provided by Lemma \ref{L4} and Cauchy's inequality to find
$$ II \leq C_T \left ( \Vert v_0^{3/2} \partial_x f(t) \Vert_2^2 +  \left (1+ \frac{1}{\eps} \right )  \Vert v_0^2 f(t) \Vert_2^2 + \eps \Vert v_0 \grad_v  \partial_x f(t) \Vert_2^2 \right ).$$
We note that for $\eps$ sufficiently small, the last term can be controlled by the final term arising in $IV$ below. Next, we integrate by parts in $III$ to find
$$ III = -\iint v_0^3 \grad_v \cdot (K | \partial_x f |^2 ) \ dv dx = 3\iint v_0^2 \hat{v} \cdot (K | \partial_x f |^2 ) \ dv dx \leq C_T \Vert v_0 \partial_x f(t) \Vert_2^2.$$
In the last term, we again integrate by parts and use $D\hat{v} = v$ along with (\ref{D1}) to find
\begin{eqnarray*}
IV & = & -\iint \left ( 3v_0^2 \hat{v} \partial_x f + v_0^3 \grad_v \partial_x f \right ) \cdot D \grad_v \partial_x f \ dv dx\\
& \leq & \Vert v_0 \partial_x f(t) \Vert_2^2 - \Vert v_0 \grad_v \partial_x f(t) \Vert_2^2.
\end{eqnarray*}
Combining the estimates, the third results follows.

To prove the last inequality, we let $\partial_v$ be either first-order derivative and compute
\begin{eqnarray*}
\frac{1}{2} \frac{d}{dt} \Vert v_0 \partial_v \partial_x f(t) \Vert_2^2 & = & \iint v_0^2 \partial_v \partial_x f 
\left [ - \hat{v}_1 \partial_v \partial_{xx} f - \partial_v \hat{v}_1 \partial_{xx} f \right. \\
 & \ & \ - \partial_v K \cdot \grad_v \partial_x f - \partial_x K \cdot \grad_v \partial_v f - \partial_v \partial_x K \cdot \grad_v f - K \cdot \grad_v \partial_v \partial_x f\\
 & \ & \left. \  + \grad_v \cdot ( (\partial_v  D) \grad_v \partial_x f ) + \grad_v \cdot (D \grad_v  \partial_v  \partial_x f ) \right ] \ dv dx\\
& = & I + II + III.
\end{eqnarray*}
Because the first term of $I$ vanishes yet again, we use Cauchy's inequality to find
$$I  = 
-\iint v_0^2 R^{-1}(v) \partial_v \partial_x f \partial_{xx} f \ dv dx 
\leq C_T \left ( \Vert v_0 \partial_v \partial_x f(t) \Vert_2^2 + \Vert \partial_{xx} f(t) \Vert_2^2 \right ).$$
To estimate $II$, we integrate by parts in the third and fourth terms below and use the bounds on fields and field derivatives (Lemmas \ref{L2} and \ref{L4}) as well as Cauchy's inequality so that
\begin{eqnarray*}
II & = & -\iint v_0^2 \partial_v \partial_x f \left [ \partial_v K \cdot \grad_v \partial_x f + \partial_x K \cdot \grad_v \partial_v f + \partial_v \partial_x K \cdot \grad_v f + K \cdot \grad_v \partial_v \partial_x f \right ] \ dv dx\\
& = &  -\iint v_0^2 R^{-1}(v) \partial_v \partial_x f  \left [ \grad_v \partial_x f  + \grad_v f \right ]\ dv dx - 
\iint v_0^2\partial_v \partial_x f   \grad_v \cdot \left ( \partial_x K \partial_v f \right )  \ dv dx\\
& \ & \  - \iint v_0^2 \grad_v \cdot \left (K \partial_v \partial_x f \right ) \ dv dx\\
& \leq & C_T \left ( \Vert v_0 \grad_v \partial_x f(t) \Vert_2^2 + \left (1 + \frac{1}{\eps} \right ) \Vert v_0 \grad_v f(t) \Vert_2^2 + \eps \Vert v_0 \grad_v \partial_v \partial_x f(t) \Vert_2^2 \right )
\end{eqnarray*}
We note that for $\eps$ sufficiently small, the last term can be controlled by the final term arising in $III$ below. 
Lastly, we estimate $III$ exactly as in the proof of the second inequality, but for $\partial_x f $ instead of $f$, to find
$$III  \leq C \Vert v_0 \grad_v \partial_x f(t) \Vert_2^2 - \Vert v_0 \grad_v \partial_v \partial_x f(t) \Vert_2^2.$$
With this, we combine the estimates, sum over all first-order $v$-derivatives, and proceed as for the second inequality, which yields the final estimate.
We note that throughout we have rescaled $\eps > 0$ by a factor of $C_T > 0$ when necessary.
\end{proof}

\begin{proof}[Lemma~\ref{v_deriv_lemma}]
For each result the proof is made more difficult because of the structure of $D$ and its derivatives, while in the case $D = \mathbb{I}$ derivatives commute with the Fokker-Planck operator and the computations are straightforward.  Let $k=2,3,4$ be given and $t \in (0,T)$.
As in the proof of the previous lemma, we will use the notation $R^\gamma(v)$ for a generic function satisfying
$|R^\gamma(v)|\leq C_T v_0^\gamma$.

Now, fix a multi-index $\alpha=(\alpha_1,\alpha_2)$ where we denote $\del_{v_1}^{\alpha_1}\del_{v_2}^{\alpha_2}$ by $\del_v^\alpha$, and consider
\begin{eqnarray*}
 \frac{1}{2}\frac{d}{dt}\norm{v_0^\gamma \del_v^\alpha f(t)}_2^2  & = &
   -\ip{v_0^\gamma \hat{v}_1\del_x \del_v^\alpha f}{v_0^\gamma \del_v^\alpha f}
   -\ip{v_0^\gamma K\cdot\grad_v \del_v^\alpha f}{v_0^\gamma \del_v^\alpha f}
   +\ip{v_0^\gamma\grad_v\cdot(D\grad_v \del_v^\alpha f)}{v_0^\gamma\del_v^\alpha f}\\
   & \ &   + \sum_{\substack{\beta+\alpha'=\alpha\\|\beta|>0}} {\alpha\choose\alpha'~\beta}\Big[
    \ip{R^{1-|\beta|+\gamma}(v)\del_x\del_v^{\alpha'}f}{v_0^\gamma \del_v^\alpha f}\\
    & \ &    +\ip{BR^{1-|\beta|+\gamma}(v)\del_{v_1}\del_v^{\alpha'}f}{v_0^\gamma \del_v^\alpha f}+\ip{BR^{1-|\beta|+\gamma}(v)\del_{v_2}\del_v^{\alpha'}f}{v_0^\gamma \del_v^\alpha f}\\
   & \ & +\ip{v_0^\gamma\grad_v\cdot(\del_v^{\beta}(D)\grad_v\del_v^{\alpha'}f)}{v_0^\gamma \del_v^\alpha f}
   \Big]\\
& =: &I+II+III+\sum_{\substack{\beta+\alpha'=\alpha\\|\beta|>0}}{\alpha\choose\alpha'~\beta} \Big[IV^1_{\alpha\beta}+IV^2_{\alpha\beta}+IV^3_{\alpha\beta}+IV^4_{\alpha\beta}\Big]
\end{eqnarray*}

For $I$, we integrate by parts in $x$ so that
$\ip{v_0^\gamma \hat{v}_1\del_x \del_v^\alpha f}{v_0^\gamma \del_v^\alpha f} 
  = -\ip{v_0^\gamma \del_v^\alpha f}{v_0^\gamma \hat{v}_1 \del_x \del_v^\alpha f}$ 
and hence the first term vanishes.
For $II$ we integrate by parts in $v$ to find
$$-2\ip{v_0^\gamma K\cdot\grad_v \del_v^\alpha f}{v_0^\gamma \del_v^\alpha
  f}=\ip{(\grad_v v_0^{2\gamma})\cdot K \del_v^\alpha f}{\del_v^\alpha
  f}+\ip{v_0^{\gamma}(\grad_v \cdot K) \del_v^\alpha f}{v_0^\gamma\del_v^\alpha f}$$
The second term vanishes by the divergence-free structure of $K$, while the first term
is bounded by field estimates so that
$$II \leq C\left (\norm{B((t)}_\infty+\norm{E(t)}_\infty \right )\norm{v_0^{\gamma}\del_v^\alpha f(t)}_2^2.$$
To estimate $III$, we integrate by parts in $v$ to find
$$\ip{v_0^\gamma\grad_v\cdot(D\grad_v \del_v^\alpha
  f)}{v_0^\gamma\del_v^\alpha f}=-\norm{v_0^\gamma D^{1/2}\grad_v\del_v^\alpha
  f(t)}_2^2-\ip{\grad_v(v_0^{2\gamma})\cdot D\grad_v\del_v^\alpha
  f}{\del_v^\alpha f}.$$
Integrating by parts again in the second of these two terms yields
$\ip{R^{\gamma-1}(v)\del_v^\alpha f}{v_0^\gamma\del_v^\alpha f}$.
So we have
$$III \leq -\norm{v_0^{\gamma-1/2}\grad_v\del_v^\alpha f(t)}_2^2 + C\norm{v_0^\gamma \del_v^\alpha f(t)}_2^2$$

Next, we estimate the terms $IV_{\alpha\beta}^1$.
If $|\alpha'|=0$ then we may use Cauchy's inequality and hence
$$IV_{\alpha\beta}^1 \leq \norm{v_0^{1-k+\gamma}\del_xf(t)}_2^2+\norm{v_0^\gamma \del_v^\alpha f(t)}_2^2.$$  
Otherwise we may write $\del_v^{\alpha'}=\del_{v_i}\del_v^{\alpha''}$ with $\vert \alpha'' \vert = \vert \alpha \vert - 2$.
Then, we integrate by parts in $v_i$ and write this term as
$$\ip{R^{1-|\beta|+\gamma}(v)\del_x\del_v^{\alpha'}f}{v_0^\gamma\del_v^\alpha f} = 
-\ip{R^{-|\beta|+\gamma}(v)\del_x\del_v^{\alpha''}f}{v_0^\gamma \del_v^\alpha f}-\ip{R^{1-|\beta|+\gamma}(v)\del_x\del_v^{\alpha''}f}{v_0^\gamma \del_{v_i}\del_v^\alpha f}$$
Applying Cauchy's inequality with $\eps > 0$ to both terms we arrive at
$$IV_{\alpha\beta}^1 \leq \frac{C}{\eps}\norm{v_0^{\gamma+1/2}\del_v^{\alpha''}\del_xf(t)}_2^2+\norm{v_0^\gamma\del_v^\alpha f(t)}_2^2
+\eps\norm{v_0^{\gamma-1/2}\del_{v_i}\del_v^\alpha f(t)}_2^2$$
and we can choose $\eps$ small enough so that the last term here is
absorbed by the first term in the estimate of $III$.
Both $IV_{\alpha\beta}^2$ and $IV_{\alpha\beta}^3$ possess the form
$\ip{BR^{1-|\beta|+\gamma}(v)\del_{v_j}\del_v^{\alpha'}f}{v_0^\gamma \del_v^\alpha f}.$
Hence, after applying Cauchy's inequality we find
\begin{align}
IV_{\alpha\beta}^2+IV_{\alpha\beta}^3&\leq
\norm{B(t)}_\infty\left(\norm{R^{1-|\beta|+\gamma}(v)\del_{v_j}\del_v^{\alpha'}f(t)}_2^2+\norm{v_0^\gamma\del_v^\alpha f(t)}_2^2\right)\\
&\leq
\norm{B(t)}_\infty\left(\norm{v_0^\gamma\del_v^\alpha f(t)}_2^2+\sum_{1 \leq |\alpha|<k}\norm{v_0^{\gamma+|\alpha|-k}\del_v^\alpha f(t)}_2^2 \right)
\end{align}

To estimate $IV_{\alpha\beta}^4$ we must consider cases.
If $|\alpha'|<|\alpha|-1$, we use Cauchy's inequality to find
\begin{eqnarray*}
IV_{\alpha\beta}^4 & \leq & C \Big(\norm{v_0^{\gamma-|\beta|}\del_{v_i}\del_v^{\alpha'}f(t)}_2^2+\norm{v_0^{\gamma+1-|\beta|}\del_{v_i}\del_{v_j}\del_v^{\alpha'} f(t)}_2^2+\norm{v_0^\gamma\del_v^\alpha f(t)}_2^2\Big)\\ & \leq & C\left(\norm{v_0^\gamma\del_v^\alpha f(t)}_2^2+\sum_{1 \leq |\alpha|<k}\norm{v_0^{\gamma+|\alpha|-k}\del_v^\alpha f(t)}_2^2\right)
\end{eqnarray*}
If $|\alpha'|=|\alpha|-1$, suppose $\del_{v_i}\del_v^{\alpha'}=\del_v^\alpha$. Then the terms involving $\del_{v_i}\del_v^{\alpha'}$ can be handled using Cauchy-Schwarz.
The terms involving $\del_{v_iv_j}^2\del_v^{\alpha'}f$, after integration by parts, are bounded by $\norm{v_0^\gamma\del_v^\alpha f(t)}_2^2$.  


Collecting the estimates, summing over all $\alpha$ with $| \alpha | = k$, and writing 
$$ \Vert v_0^\gamma \grad_v^k f(t) \Vert_2^2 = \sum_{|\alpha|=k} \norm{v_0^\gamma \del^\alpha f(t)}_2^2$$ we find
\begin{eqnarray*}
 \Vert v_0^\gamma \grad_v^k f(t) \Vert_2^2 & \leq & C \left (  \Vert v_0^\gamma \grad_v^k f(t) \Vert_2^2 + \sum_{j=1}^{k-1} \norm{v_0^{\gamma+j-k}\grad_v^j f(t)}_2^2
 + \norm{v_0^{\gamma+1/2} \del_x \grad_v^{k+2}f(t)}_2^2  \right )\\
& \ & \  - (1-\eps) \norm{v_0^{\gamma-1/2}\grad_v^{k+1} f(t)}_2^2
\end{eqnarray*}
which proves the first result.

Next, we turn to the second result.  Let $\partial_v^2$ be any second-order $v$-derivative.  We compute
\begin{eqnarray*}
\frac{1}{2}\frac{d}{dt} \Vert v_0^{1/2} \partial_v^2 \partial_x f(t) \Vert_2^2 & = & \iint v_0 \partial_v^2 \partial_x f \left [
- \partial_v^2 \left (\hat{v}_1 \partial_{xx} f \right )
- \partial_v^2 \partial_x \left (K \cdot \grad_v f \right )
+ \partial_v^2 \left (\grad_v \cdot (D \grad_v \partial_xf ) \right ) \right ] \ dv dx\\
& = & I + II + III.
\end{eqnarray*}
As usual, one of the terms in $I$ vanishes. So, we integrate by parts in the latter term below and use Cauchy's inequality with $\eps > 0$ to find
\begin{eqnarray*}
I & = &  - \iint v_0  \partial_v^2 \partial_x f \left [ R^{-2}(v) \partial_{xx} f + 2R^{-1}(v) \partial_v \partial_{xx} f\right ] \ dv dx\\
& \leq & C \left ( \Vert \partial^2_v \partial_xf(t) \Vert_2^2 +  \left (1 + \frac{1}{\eps} \right ) \Vert \partial_{xx} f(t) \Vert_2^2 + \eps \Vert \partial^3_v \partial_x f(t) \Vert_2^2 \right).
\end{eqnarray*}
We note that for $\eps$ sufficiently small, the last term can be controlled by the final term arising in $III$ below. To estimate $II$, we integrate by parts in the third and last terms below, use the control of field and field derivative terms guaranteed by Lemmas \ref{L2} and \ref{L4}, and utilize Cauchy's inequality so that
\begin{eqnarray*}
II & = &  - \iint v_0  \partial_v^2 \partial_x f \left [ R^{-2}(v) \grad_v f + 2R^{-1}(v) \grad_v \partial_v  f + R^0(v) \grad_v \partial^2_v  f \right. \\
& \ & +  \left.  R^{-2}(v) \grad_v \partial_x f + 2R^{-1}(v) \grad_v \partial_v \partial_x  f + K \cdot \grad_v \partial^2_v \partial_x   f  \right ] \ dv dx\\
& \leq & C_T \biggl( \Vert v_0^{1/2} \grad^2_v \partial_xf(t) \Vert_2^2 +  \Vert \nabla_v f(t) \Vert_2^2  + \Vert \grad_v \partial_x f(t) \Vert_2^2\\
& \ & + \left.  \left( 1 + \frac{1}{\eps} \right ) \Vert v_0 \nabla_v \partial_v f(t) \Vert_2^2 + \eps \Vert \nabla_v \partial^2_v \partial_x f(t) \Vert_2^2  \right).
\end{eqnarray*}
Again, for $\eps$ sufficiently small, the last term can be controlled by the final term arising in $III$ below.
We integrate by parts, then use aforementioned properties of $D$ and Cauchy's inequality with $\eps > 0$ in $III$ to find
\begin{eqnarray*}
III & = &  - \iint \left [ \hat{v} \partial^2_v \partial_x f + v_0 \grad_v \partial^2_v \partial_x f \right ] [  \partial^2_v D \grad_v \partial_x f  + 2\partial_v D \grad_v \partial_v \partial_x f + D \grad_v \partial^2_v \partial_x f] \ dv dx\\
& = &  - \iint \left [ \hat{v} \partial^2_v \partial_x f + v_0 \grad_v \partial^2_v \partial_x f \right ] [  R^{-1}(v) \grad_v \partial_x f  + R^0(v) \grad_v \partial_v \partial_x f + D \grad_v \partial^2_v \partial_x f] \ dv dx\\
& \leq & C \left ( \Vert \partial^2_v \partial_x f(t) \Vert_2^2 + \left ( 1 + \frac{1}{\eps} \right ) \Vert \grad_v \partial_x f(t) \Vert_2^2 + \Vert \grad_v \partial_v \partial_x f(t) \Vert_2^2 \right ) - (1- \eps) \Vert \grad_v \partial^2 _v \partial_x f(t) \Vert_2^2.
\end{eqnarray*}
Finally, we collect these estimates, so that 
\begin{eqnarray*}
I + II + III & \leq & C_T\left ( \Vert \partial^2_v \partial_xf(t) \Vert_2^2 +  \Vert \partial_{xx} f(t) \Vert_2^2 +
\Vert v_0^{1/2} \grad^2_v \partial_xf(t) \Vert_2^2 +  \Vert \nabla_v f(t) \Vert_2^2 +  \Vert \nabla_v \partial_v f(t) \Vert_2^2 \right.\\
& \ & \ \left.  + \Vert \grad_v \partial_x f(t) \Vert_2^2 +
\Vert \partial^2_v \partial_x f(t) \Vert_2^2 \right)
- (1- C_T\eps) \Vert \grad_v \partial^2 _v \partial_x f(t) \Vert_2^2.
 \end{eqnarray*}
Then, we use  $\Vert \partial^2_v \partial_x f(t) \Vert_2^2 \leq \Vert \grad^2_v  \partial_xf(t) \Vert_2^2$, sum over all $v$-derivatives to arrive at an estimate on $\frac{d}{dt} \Vert v_0^{1/2} \grad^2_v \partial_x f(t) \Vert_2^2$, and the claim then follows.
As for Lemma \ref{dissipative} we have rescaled $\eps > 0$ by a factor of $C_T > 0$ where necessary.
\end{proof}

%

\begin{proof} [Lemma \ref{L8}]
We will prove the result in a hierarchical fashion by building pairs of consecutive terms and adding higher-order derivatives as we go. 
To begin the proof, we consider $t \in (0,T)$ and define
$$M_1(t) = \Vert v_0^2 f(t) \Vert_2^2 + \frac{1}{2}t \Vert v_0^{3/2} \grad_v f(t) \Vert_2^2
+ \frac{1}{8} t^2 \Vert v_0 \grad^2_v f(t) \Vert_2^2$$
and differentiate to find
\begin{eqnarray*}
M_1'(t) & = & \frac{d}{dt} \Vert v_0^2 f(t) \Vert_2^2 + \frac{1}{2}t  \frac{d}{dt}  \Vert v_0^{3/2} \grad_v f(t) \Vert_2^2
+ \frac{1}{8} t^2  \frac{d}{dt}  \Vert v_0 \grad^2_v f(t) \Vert_2^2\\
& \ & \ + \frac{1}{2} \Vert v_0^{3/2} \grad_v f(t) \Vert_2^2
+ \frac{1}{4} t \Vert v_0 \grad^2_v f(t) \Vert_2^2
\end{eqnarray*}
Using Lemma \ref{dissipative}, we find
$$\frac{d}{dt} \Vert v_0^2 f(t) \Vert_2^2 \leq C_T \Vert v_0^2 f(t) \Vert_2^2  - \Vert v_0^{3/2} \grad_v f(t) \Vert_2^2$$
and
$$\frac{d}{dt} \Vert v_0^{3/2} \grad_v f(t) \Vert_2^2 \leq C_T  \left( \Vert v_0^{3/2} \grad_v f(t) \Vert_2^2 
+ \Vert v_0 \partial_x f(t) \Vert_2^2 \right )  - \Vert v_0 \grad^2_v f(t) \Vert_2^2.$$
Additionally, applying the first result of Lemma \ref{v_deriv_lemma} for $\gamma = 1$, $k=2$ we find
for any $\eps > 0$ sufficiently small
$$\frac{d}{dt} \Vert v_0 \grad^2_v f(t) \Vert_2^2 \leq C_T  \left( \Vert v_0 \grad^2_v f(t) \Vert_2^2
+ \Vert v_0^{3/2} \grad_v f(t) \Vert_2^2 
+ \Vert v_0^{3/2} \partial_x f(t) \Vert_2^2 \right )  - (1-\eps) \Vert v_0^{1/2} \grad^3_v f(t) \Vert_2^2.$$
We combine these results, use the bounds on $x$-derivatives of the particle distribution (Lemma \ref{L6}), and choose $\eps < 1$ to find
\begin{eqnarray*}
M_1'(t) & \leq & C_T  \left ( M_1(t) + \Vert v_0^{3/2} \partial_x f(t) \Vert_2^2 \right ) - \frac{1}{2} \Vert v_0^{3/2} \grad_v f(t) \Vert_2^2\\
& \ & \ - \frac{1}{4}t \Vert v_0 \grad^2_v f(t) \Vert_2^2
- \frac{(1-\eps)}{8} t^2 \Vert v_0^{1/2} \grad^3_v f(t) \Vert_2^2\\
& \leq & C_T  \left ( 1+  M_1(t) \right )
\end{eqnarray*}
Thus, by Gronwall's inequality, we conclude 
$$M_1(t) \leq C_T M_1(0) = C_T \Vert v_0^2 f^0 \Vert_2^2$$
Hence, for $t \in (0,T)$
\begin{equation}
\label{M1a}
\Vert v_0^{3/2} \grad_v f(t) \Vert_2^2 \leq \frac{C_T}{t}
\end{equation}
and
\begin{equation}
\label{M1b}
\Vert v_0 \grad^2_v f(t) \Vert_2^2 \leq \frac{C_T}{t^2}.
\end{equation}

Next, define
$$M_2(t) = \Vert v_0^{3/2} \partial_x f(t) \Vert_2^2 + \frac{1}{2}t \Vert v_0 \grad_v \partial_x f(t) \Vert_2^2$$
and differentiate to find

$$M_2'(t) = \frac{d}{dt} \Vert v_0^{3/2} \partial_x f(t) \Vert_2^2+ \frac{1}{2}t  \frac{d}{dt}  \Vert v_0 \grad_v \partial_x f(t) \Vert_2^2
 + \frac{1}{2}  \Vert v_0 \grad_v \partial_x f(t) \Vert_2^2$$
Using Lemma \ref{dissipative}, we find
$$\frac{d}{dt} \Vert v_0^{3/2} \partial_x f(t) \Vert_2^2 \leq C_T \left (\Vert v_0^{3/2} \partial_x f(t) \Vert_2^2  + \Vert v_0^2 f(t) \Vert_2^2 \right )-  (1-\eps) \Vert v_0 \grad_v \partial_x f(t) \Vert_2^2$$
and 
$$ \frac{d}{dt} \Vert v_0 \grad_v \partial_x f(t) \Vert_2^2  \leq C_T  \left( \Vert v_0 \grad_v \partial_x f(t) \Vert_2^2 
+ \Vert \partial_{xx} f(t) \Vert_2^2  + \Vert v_0^{3/2} \grad_v f(t) \Vert_2^2  \right ) - (1-\eps) \Vert v_0^{1/2} \grad^2_v \partial_x f(t) \Vert_2^2.$$
Combining these results while using the $L^2$-bounds on second $x$-derivatives of the density (Lemma \ref{L7}) and (\ref{M1a}), we find
\begin{eqnarray*}
M_2'(t) & \leq & C_T  \left ( M_2(t) + \Vert v_0^2 f(t) \Vert_2^2
+ t \Vert v_0^{3/2} \grad_v f(t) \Vert_2^2
+ \Vert \partial_{xx} f(t) \Vert_2^2 \right )\\
& \ & \ - \left (\frac{1}{2}-\eps \right ) \Vert v_0 \grad_v \partial_x f(t) \Vert_2^2
 - \frac{1}{2}\left (1-\eps \right ) t \Vert v_0^{1/2} \grad^2_v \partial_x f(t) \Vert_2^2\\
& \leq & C_T  \left ( 1+  M_2(t) \right )
\end{eqnarray*}
Thus, for $\eps < 1/2$, we use Gronwall's inequality to conclude 
$$M_2(t) \leq C_T M_2(0) = C_T \Vert v_0^{3/2} \partial_x f^0 \Vert_2^2$$
Hence, for $t \in (0,T)$
\begin{equation}
\label{M2}
\Vert v_0 \grad_v \partial_x f(t) \Vert_2^2 \leq \frac{C_T}{t}.
\end{equation}

Building onto previous terms, we next define
$$M_3(t ) = M_1(t) + \frac{1}{48} t^3 \Vert v_0^{1/2} \grad_v^3 f(t) \Vert_2^2.$$
Hence, using the estimate of $M_1'(t)$ we find
\begin{eqnarray*}
M_3'(t) & =& M_1'(t) + \frac{1}{16}t^2\Vert v_0^{1/2} \grad_v^3 f(t) \Vert_2^2
+ \frac{1}{48} t^3 \frac{d}{dt} \Vert v_0^{1/2} \grad_v^3 f(t) \Vert_2^2\\
& \leq & C_T ( 1 + M_1(t) ) + \left ( \frac{1}{16} - \frac{(1-\eps)}{8} \right ) t^2 \Vert v_0^{1/2} \grad_v^3 f(t) \Vert_2^2
+ \frac{1}{48} t^3 \frac{d}{dt} \Vert v_0^{1/2} \grad_v^3 f(t) \Vert_2^2\\
& \leq & C_T ( 1 + M_1(t) ) + \frac{1}{48} t^3 \frac{d}{dt} \Vert v_0^{1/2} \grad_v^3 f(t) \Vert_2^2
\end{eqnarray*}
for $\eps < 1/2$.
By the first result of Lemma \ref{v_deriv_lemma} with $\gamma = 1/2$ and $k=3$, we find for any $\eps > 0$ sufficiently small
$$\begin{aligned}
\frac{d}{dt} \Vert v_0^{1/2} \grad_v^3 f(t) \Vert_2^2 & \leq 
C_T \left ( \Vert v_0^{1/2} \grad_v^3 f(t) \Vert_2^2 + \Vert v_0^{3/2} \grad_v f(t) \Vert_2^2
+ \Vert v_0 \grad_v^2 f(t) \Vert_2^2  + \Vert v_0 \grad_v \partial_x f(t) \Vert_2^2  \right )\\
& \ - (1-\eps) \Vert  \grad_v^4 f(t) \Vert_2^2 
\end{aligned}$$
Therefore, using the previous bounds obtained from (\ref{M1a}), (\ref{M1b}), and (\ref{M2}), we have
$$M_3'(t) \leq C_T  \left (1 +  M_3(t) \right) - \frac{1}{48}(1-\eps)t^3 \Vert \grad_v^4 f(t) \Vert_2^2.$$
Since $\eps < 1/2$ Gronwall's inequality implies
\begin{equation}
\label{M3}
\Vert v_0^{1/2} \grad_v^3 f(t) \Vert_2^2 \leq \frac{C_T}{t^3}
\end{equation}
for $t \in (0,T)$.

Again building onto previous terms, we next define
$$M_4(t) = M_2(t) + \frac{1}{8}t^2 \Vert v_0^{1/2} \grad^2_v \partial_x f(t) \Vert_2^2$$
so that
$$M_4'(t) = M_2'(t) +   \frac{1}{4}t \Vert v_0^{1/2} \grad^2_v \partial_x f(t) \Vert_2^2 +  \frac{1}{8}t^2 \frac{d}{dt} \Vert v_0^{1/2} \grad^2_v \partial_x f(t) \Vert_2^2$$
Using the second result of Lemma \ref{v_deriv_lemma} along with the bound on $\Vert \partial_{xx}f(t) \Vert_2^2$ from Lemma \ref{L7} and the previous bounds obtained from (\ref{M1a}), (\ref{M1b}), and (\ref{M2}), we find for $\eps > 0$ sufficiently small
\begin{eqnarray*}
\frac{d}{dt} \Vert v_0^{1/2} \grad_v^2 \partial_x f(t) \Vert_2^2 & \leq & C_T \biggl(\Vert v_0^{1/2} \grad_v^2 \partial_x f(t) \Vert_2^2
+ \Vert v_0 \grad_v \partial_x f(t) \Vert_2^2
+  \sum_{j=1}^2 \norm{v_0^{\frac{4-j}{2}} \grad_v^j f (t)}_2^2 \\
& \ & \ + \Vert \partial_{xx} f(t) \Vert_2^2  \biggr)- (1- \eps) \norm{\grad_v^3 \partial_x f (t)}_2^2 \\
&\leq & C_T \Big(\frac{1}{t^2} + \Vert v_0^{1/2} \grad_v^2 \partial_x f(t) \Vert_2^2 \Big).
\end{eqnarray*}
Hence, we incorporate this and use the estimate of $M_2'(t)$ to find
\begin{eqnarray*}
M_4'(t) & \leq & C_T (1 + M_4(t) ) - \left (\frac{1}{4} - \frac{1 -\eps}{2} \right ) t \Vert v_0^{1/2} \grad^2_v \partial_x f(t) \Vert_2^2 \\
& \leq & C_T (1 + M_4(t) ) 
\end{eqnarray*}
and upon choosing $\eps < 1/2$ an application of Gronwall's inequality yields the bound
\begin{equation}
\label{M4}
\Vert v_0^{1/2} \grad^2_v \partial_x f(t) \Vert_2^2 \leq \frac{C_T}{t^2}
\end{equation}
for $t \in (0,T)$.
Finally, to obtain bounds on fourth-order $v$-derivatives of the density, we define
$$M_5(t) = M_3(t) + \frac{1}{2^4 4!} t^4 \Vert \grad^4_v f(t) \Vert_2^2$$
so that
$$M_5'(t) = M_3'(t) + \frac{1}{96} t^3 \Vert \grad^4_v f(t) \Vert_2^2 +  \frac{1}{2^4 4!} t^4 \frac{d}{dt}\Vert \grad^4_v f(t) \Vert_2^2.$$ 
Using Lemma \ref{v_deriv_lemma} one final time with $\gamma = 0$ and $k=4$ and utilizing the bounds obtained from (\ref{M1a})-(\ref{M4}), we find
$$ \frac{d}{dt}\Vert \grad^4_v f(t) \Vert_2^2 \leq C_T \left ( \frac{1}{t^3} + \Vert \grad^4_v f(t) \Vert_2^2 \right ).$$
Applying this to $M_5(t)$ and using the estimate of $M_3'(t)$, we see
$$M_5'(t) \leq C_T(1 + M_5(t)) + \left ( \frac{1}{96} - \frac{(1-\eps)}{48} \right ) t^3 \Vert \grad_v^4 f(t) \Vert_2^2.$$
and choosing $\eps < 1/2$ this implies 
$$\Vert \grad^4_v f(t) \Vert_2^2 \leq \frac{C_T}{t^4}$$
for $t \in (0,T)$.
Lastly, combining the estimates above, the proof of the lemma is complete.

We remark that this same argument can be applied to the eight term power series expansion
$$\sum_{k=0}^4 \frac{t^k}{2^k k!} \norm{v_0^{(4-k)/2} \grad_v^k f(t)}_2^2 + \sum_{k=0}^2 \frac{t^k}{2^kk!} \norm{v_0^{(3-k)/2} \grad_v^k \partial_x f(t)}_2^2$$
in order to arrive at an identical result.
However, the above argument is perhaps clearer.  Also, estimates of higher derivatives can be obtained if one imposes additional spatial regularity on the density and field terms, as this requires control of second-order field derivatives in $L^\infty$.  
\end{proof}

\section{Proof of Theorem \ref{T1}}
To conclude the paper, we utilize the previous lemmas to sketch the proof of Theorem \ref{T1}.

\begin{proof}
As is typical, the proof utilizes a standard iterative argument.  We define a sequence of solutions to the corresponding linear equations and show that it must converge to a solution of the nonlinear system (\ref{RVMFP}).  Define the initial iterates in terms of the given initial data
$$\begin{gathered}
f^0(t,x,v) =  f^0(x,v),\\
E_2^0(t,x) = E_2^0(x)\\
B^0(t,x) =  B^0(x).
\end{gathered}$$
Additionally, for every $n \in \bfN$, given $E_1^n, E_2^n, B \in L^\infty([0,\infty); H^2(\bfR))$ we obtain $f^n \in L^\infty([0,\infty) \times \bfR^3)$ by solving the linear initial value problems
\begin{equation}
\label{fnIC}
\left\{
\begin{gathered}
\partial_t f^n + \hat{v}_1 \partial_x f^n + K^{n-1}\cdot \nabla_v f^n = \nabla_v \cdot ( D \nabla_v f^n) \\
f^n(0,x,v) =  f^0(x,v),
\end{gathered}
\right.
\end{equation}
where $$ K^n = \langle E^n_1 + \hat{v}_2 B^n, E^n_2 - \hat{v}_1 B^n \rangle$$
and the fields satisfy
\begin{equation}
\label{FieldnIC}
\left\{
\begin{gathered}
\partial_t E^n_2 +\partial_x B^n =  - j_2, \quad \partial_t B^n + \partial_x E^n_2 = 0\\ 
E^n_1 = \int_{-\infty}^x \left ( \int  f^n(t,y,v) \ dv - \phi(y) \right ) \ dy\\
E_2^n(0,x) = E_2(0,x)\\
B^n(0,x) = B(0,x)
\end{gathered}
\right.
\end{equation}
respectively.  
Let $T > 0$ be given and $(f^n, E^n, B^n)$ be a sequence of weak solutions to the above linear system on $[0,T]$.  Using the assumptions on initial data, we apply the estimates of Section $2$ and find $E_2^n$ and $B^n$ converge (up to a subsequence) weakly in $L^\infty( [0,T]; H^1(\bfR))$ to functions $E_2$ and $B$, respectively.  Then, we proceed by estimating successive differences of iterates (e.g., see \cite{Lai}). First, we use (\ref{E2B}) and the linearity of the transport equation to find
$$ \Vert K^{n}(t) - K^{n-1}(t) \Vert_\infty \leq Ct \sup_{s\in [0,t]} \Vert v_0^a f^n(s) -v_0^a  f^{n-1}(s) \Vert_\infty.$$
Next, we write the Vlasov equation for the difference of consecutive iterates and use (\ref{V}) and Lemmas \ref{L1}, \ref{L2}, and \ref{L3} to conclude
$$\Vert v_0^a f^{n+1}(t) -  v_0^a f^n(t)  \Vert_\infty \leq C_T \int_0^t \left ( \Vert K^{n}(s) - K^{n-1}(s) \Vert_\infty  + \Vert v_0^a f^{n+1}(s) - v_0^a  f^n(s) \Vert_\infty \right ) \ ds$$
and thus
\begin{equation}
\label{successive}
\Vert v_0^a f^{n+1}(t) - v_0^a  f^n(t)  \Vert_\infty \leq C_T \int_0^t \sup_{\tau \in [0,s] }\Vert v_0^a f^{n}(\tau) -  v_0^a f^{n-1}(\tau)  \Vert_\infty \ ds.
\end{equation}
It follows from this estimate that $f^n$ converges strongly to a function $f$ in $L^\infty ( [0,T] \times \bfR^3)$.  Similar estimates can be used to show $f \in L^\infty([0,T]; L^1(\bfR^3))$ as in \cite{Degond}. It can then be shown that these limiting functions satisfy (\ref{RVMFP}) in the weak sense. Applying the regularizing estimates, we find $ f \in L^\infty ( (0,T); H_x^2(\bfR; H_v^4(\bfR^2)) )$.  By the Sobolev Embedding Theorem, $H^2(\bfR) \subset C_b^1(\bfR)$ and $H^4(\bfR^2) \subset C_b^2(\bfR^2)$.  Thus we find $f$, $E_2$, and $B$ possess a continuous partial derivative in $x$, and $f$ possesses two continuous partial derivatives in either $v$ component.  Using the Vlasov and transport equations, we see that $\partial_t B$, $\partial_t E_2$, and $\partial_t f$ are all continuous.  Hence, we find $f \in C^1( (0,T) \times \bfR; C^2(\bfR^2))$ and $E_2, B \in C^1( (0,T) \times \bfR)$.  Finally, from the regularity of $f$ we deduce $E_1 \in C^1((0,T) \times \bfR)$ as well.   Of course, with this additional regularity we conclude that the triple $(f,E_2,B)$ is, in fact, a classical solution of (\ref{RVMFP}).

The uniqueness of solutions follows from another standard argument.  We consider the difference of solutions $$h(t,x,v) = v_0^a (f_1(t,x,v) - f_2(t,x,v))$$
where $f_1$ and $f_2$ are any two solutions of (\ref{RVMFP}) which share the same initial data, and 
we derive the same estimate (\ref{successive}) for $h$, namely
$$ \Vert h(t) \Vert_\infty \leq C_T \int_0^t \sup_{\tau \in [0,s] }\Vert h(\tau)  \Vert_\infty \ ds.$$
After an application of Gronwall's inequality, it follows that $f_1 \equiv f_2$ and solutions are unique.

From the proof of this theorem and the previous lemmas, additional classical regularity of solutions can be obtained by imposing that further spatial derivatives of the initial data $f^0$, $E_2^0$, and $B^0$ belong to $L^2(\bfR^3)$. 
\end{proof}

\bibliographystyle{acm}
\bibliography{SDPrefs}

\end{document}